\RequirePackage{fix-cm}
\documentclass[smallextended]{svjour3}       
\smartqed  
\usepackage{graphicx}
 \usepackage{mathptmx}      
\usepackage{latexsym}
\usepackage{amsmath,amssymb,txfonts}
\usepackage{amssymb}
\usepackage{amsmath}
\usepackage{mathrsfs}
\usepackage{amsmath,amssymb}
\usepackage{color}
\usepackage{extarrows}

\allowdisplaybreaks[4]
\begin{document}

\title{Boundedness of operators generated by fractional semigroups associated with Schr\"odinger operators on Campanato type spaces via $T1$ theorem
\thanks{P. Li was financially supported by the National Natural Science Foundation of China (No. 12071272) and   Shandong Natural Science Foundation of China (Nos. ZR2020MA004, ZR2017JL008). C. Zhang was supported by the National Natural Science Foundation of China (No. 11971431), the Zhejiang Provincial Natural Science Foundation of China(Grant No. LY18A010006) and
the first Class Discipline of Zhejiang-A(Zhejiang Gongshang University-Statistics).}
\thanks{Pengtao Li is the corresponding author.}
}
\subtitle{}

\titlerunning{Boundedness of operators on Campanato type spaces}        

\author{Zhiyong Wang \and    Pengtao Li*      \and    Chao Zhang      }

\authorrunning{Z. Wang \and P. Li\and C. Zhang}

\institute{Zhiyong Wang \at
              School of Mathematics and Statistics, Qingdao University, Qingdao, 266071, China \\
              \email{zywang199703@163.com}           
           \and
           Pengtao Li \at
              School of Mathematics and Statistics, Qingdao University, Qingdao, 266071, China\\
              \email{ptli@qdu.edu.cn}
              \and
              Chao Zhang \at
 School of Statistics and Mathematics, Zhejiang Gongshang University, Hangzhou, 310018, China\\
 \email{zaoyangzhangchao@163.com}             }

\date{Received: date / Accepted: date}

\maketitle

\begin{abstract}
Let $\mathcal{L}=-\Delta+V$ be a Schr\"{o}dinger operator, where the nonnegative potential $V$ belongs to the reverse H\"{o}lder class $B_{q}$.
By the aid of the subordinative formula, we estimate the regularities of the fractional heat semigroup,  $\{e^{-t\mathcal{L}^{\alpha}}\}_{t>0},$
associated with $\mathcal{L}$. As an application,  we obtain the $BMO^{\gamma}_{\mathcal{L}}$-boundedness of the maximal function, and the Littlewood-Paley $g$-functions  associated with $\mathcal{L}$  via $T1$ theorem, respectively.
 \end{abstract}
\section{Introduction}
In the research of harmonic analysis and partial differential equations, the maximal operators and Littelwood-Paley $g$-functions paly an important role and were investigated by many mathematicians extensively.
For any integrable function $f$ on $\mathbb R^{n}$, the Hardy-Littlewood maximal operator is defined as
$$M(f)(x):=\sup_{Q}\frac{1}{|Q|}\int_{Q}|f(y)|dy,$$
where the supremum is taken over all cubes $Q\subset \mathbb R^{n}$.  For $f\in BMO(\mathbb R^{n})$, Bennett-DeVore-Sharpley proved in \cite{BDS} that $M(f)$ is either infinite or belongs to $BMO(\mathbb R^{n})$. The boundedness result in \cite{BDS} can be extended to other maximal operators. For example, let $-\Delta$ be the Laplace operator: $\displaystyle \Delta=\sum^{n}_{i=1}\frac{\partial^{2}}{\partial x_{i}^{2}}$. Denote by $M_{\Delta}$ and $g$ the maximal operator and Littlewood-Paley $g$-function generated by the heat semigroup $\{e^{-t(-\Delta)}\}_{t>0}$, respectively, i.e.,
\begin{equation}\label{eq-1.1}
\left\{\begin{aligned}
\displaystyle &M_{\Delta}(f)(x):=\sup_{t>0}|e^{-t(-\Delta)}(f)(x)|;\\
&g(f)(x):=\Bigg(\int_{0}^{\infty}|e^{-t(-\Delta)}(f)(x)|^{2}\frac{dt}{t^{n+1}}\Bigg)^{1/2}.
\end{aligned}\right.
\end{equation}
Due to the mean value on ``large" cubes may be infinite, the $BMO(\mathbb R^{n})$-boundedness of $M_{\Delta}$ or $g$ holds if  $M_{\Delta}(f)<\infty$ or $g(f)<\infty$ for $f\in BMO(\mathbb R^{n})$.

However, if the Laplacian $-\Delta$ is replaced by other second-order differential operators, the situation becomes different. Consider the Schr\"{o}dinger $\mathcal{L}=-\Delta+V$ in $\mathbb{R}^{n},\ \ n\geq 3,$
where  $V$ is a nonnegative potential belonging to the reverse H\"{o}lder class $B_{q}$ for some $q>n/2$. Here a nonnegative potential $V$ is said to belong to $B_{q}$ if there exists $C>0$ such that for  every ball $B$,
$$\Bigg(\frac{1}{|B|}\int_{B}V^{q}(x)dx\Bigg)^{1/q}\leq \frac{C}{|B|}\int_{B}V(x)dx.$$
In \cite{BDS}, the authors pointed out that for $f\in BMO(\mathbb R^{n})$ a supremum of averages of $f$ over ``large" cubes may be infinite, see \cite[page 610]{BDS}. In 2005, Dziuba\'{n}ski et.al. \cite{DGMTZ} proved the square functions associated with Schr\"odinger operators are bounded on the BMO type space $BMO_{\mathcal{L}}(\mathbb R^{n})$ related with $\mathcal{L}$ which  is distinguished from the case of $BMO(\mathbb R^{n})$. See also \cite{LL} and \cite{wyx} for similar results in the setting of Heisenberg groups and stratified Lie groups.

Let $H=-\Delta+|x|^{2}$ be the harmonic oscillator. In \cite{bet}, Betancor et al. introduced a $T1$ criterion for Calder\'{o}n-Zygmund operators related to $H$ on the BMO type  space $BMO_{H}(\mathbb{R}^{n})$. Later, Ma et al. in \cite{ma2} generalized the $T1$ criterion to the case of Campanato type spaces $BMO_{\mathcal{L}}^{\gamma}(\mathbb{R}^{n})$ related with $\mathcal{L}$. As applications, the authors in  \cite{ma2} proved that the maximal operators associated with the heat
semigroup $\{e^{-t\mathcal{L}}\}_{t>0}$ and with the generalized Poisson operators
$\{P^{\sigma}_{t}\}_{t>0}(0<\sigma<1)$, the Littlewood-Paley g-functions given in terms of the heat and the Poisson semigroups are bounded on $BMO^{\gamma}_{\mathcal{L}}(\mathbb R^{n})$.

Notice that for $\sigma\in(0,1)$, the generalized Poisson operator
$\{P^{\sigma}_{t}\}_{t>0}$ is expressed as
\begin{equation}\label{eq-1.2}
P^{\sigma}_{t}(f)(x):=\frac{t^{2\sigma}}{4^{\sigma}\Gamma(\sigma)}\int^{\infty}_{0}e^{-\frac{t^{2}}{4r}}e^{-r\mathcal{L}}(f)\frac{dr}{r^{1+\sigma}}.
\end{equation}
Specially, for $\sigma=1/2$, $\{P^{1/2}_{t}\}_{t>0}$ is corresponding to the Poisson semigroup $\{e^{-t\mathcal{L}^{1/2}}\}_{t>0}$ associated with $\mathcal{L}$.
The main purpose of this paper is to derive the pointwise estimate and regularity properties of the fractional heat semigroup $\{e^{-t\mathcal{L}^{\alpha}}\}_{t>0}$, $\alpha>0$, to  prove the boundedness of the maximal function and the Littlewood-Paley $g$-functions generated by $\{e^{-t\mathcal{L}^{\alpha}}\}_{t>0}$  on $BMO^{\gamma}_{\mathcal{L}}(\mathbb R^{n})$ via $T1$ theorem, respectively.

When $\mathcal{L}=-\Delta$, the kernels of the fractional heat semigroup $\{e^{-t(-\Delta)^{\alpha}}\}_{t>0}$ can be defined via the Fourier transform, i.e.,
\begin{equation}\label{1.4}
K_{\alpha, t}(x)=(2\pi)^{-n/2}\int_{\mathbb R^n}e^{ix\cdot \xi-t|\xi|^{2\alpha}}d\xi.
\end{equation}
For $\mathcal{L}=-\Delta+V$ with $V>0$, the kernels of fractional heat semigroups $\{e^{-t\mathcal{L}^{\alpha}}\}_{t>0}$, $\alpha\in(0,1)$, can not be defined via (\ref{1.4}).
However, for $\alpha>0$, the subordinative formula (cf.\cite{gri}) indicates that
\begin{equation}\label{p-1}
  K^{\mathcal{L}}_{\alpha,t}(x,y)=\int^{\infty}_{0}\eta^{\alpha}_{t}(s)K^{\mathcal{L}}_{s}(x,y)ds,
\end{equation}
where  $\eta^{\alpha}_{t}(\cdot)$ is a continuous function  on $(0,\infty)$ satisfying (\ref{eq-2.1}) below.
In \cite{Li2}, the identity (\ref{p-1}) was applied to estimate $K^{\mathcal{L}}_{\alpha,t}(\cdot,\cdot)$ via the heat kernel $K^{\mathcal{L}}_{t}(\cdot,\cdot)$, see Proposition \ref{pro2.51}. Specially, for $\alpha=1/2$, the estimates of $K^{\mathcal{L}}_{\alpha,t}(\cdot,\cdot)$ goes back to those of the Poisson kernel $P_{t}^{\mathcal{L}}(\cdot,\cdot)$, see \cite[Lemma 3.9]{duo}.

We point out that, compared with the case of $\{P^{\sigma}_{t}\}_{t>0}$, some new regularity estimates should be introduced  to prove the $BMO^{\gamma}_{\mathcal{L}}$-boundedness of the maximal function and Littlewood-Paley $g$-functions generated by $\{e^{-t\mathcal{L}^{\alpha}}\}_{t>0}$. Let $E=L^{\infty}((0,\infty), dt)$. It follows from (\ref{eq-1.2}) and the Minkowski integral inequality that
$$\|P^{\sigma}_{t}(f)\|_{E}\leq C_{\sigma}\int^{\infty}_{0}t^{2\sigma}e^{-\frac{t^{2}}{4r}}\|e^{-r\mathcal{L}}(f)\|_{E}\frac{dr}{r^{1+\sigma}}.$$
The fact that
$$\int^{\infty}_{0}t^{2\sigma}e^{-\frac{t^{2}}{4r}}\frac{dr}{r^{1+\sigma}}<\infty$$
ensures that the $BMO^{\gamma}_{\mathcal{L}}$-boundedness of the maximal function
$\sup_{t>0}|P^{\sigma}_{t}f(x)|$
can be deduced from that of the heat maximal function $\sup_{t>0}|e^{-t\mathcal{L}}f(x)|$, see \cite[Proposition 4.7]{ma2}. However, we can see from the identity (\ref{p-1}) that this method is not applicable to the case $\{e^{-t\mathcal{L}^{\alpha}}\}_{t>0}$.

In this paper, we get the following results:
\begin{itemize}
\item [$\bullet$]
In Section \ref{sec-2.1}, let $\displaystyle \nabla_{x}=\Big({\partial}/{\partial x_{1}},{\partial}/{\partial x_{2}},\ldots,{\partial}/{\partial x_{n}}\Big)$. By the perturbation theory for semigroups of operators, we deduce the pointwise estimates and the H\"older type estimates of the kernels:
$$\left\{\begin{aligned}
&\Big|\nabla_{x}(K_{t}(x-y)-K^{\mathcal{L}}_{t}(x,y))\Big|,\\
&\Big|t^{m}\partial^{m}_{t}(K_{t}(x-y)-K^{\mathcal{L}}_{t}(x,y))\Big|,\\
\end{aligned}\right.$$
see Lemmas \ref{lem1-1}, \ref{lem2-2} and Proposition \ref{pro2.5}, respectively.

\item [$\bullet$]In Section \ref{sec-2.2}, we use (\ref{1.4}) to obtain the corresponding estimates for
$$\left\{\begin{aligned}
&\Big|K^{\mathcal{L}}_{\alpha,t}(x,y)-K_{\alpha,t}(x-y)\Big|,\\
&\Big|t^{1/2\alpha}\nabla_{x}(K_{\alpha,t}(x-y)-K^{\mathcal{L}}_{\alpha, t}(x,y))\Big|,\\
&\Big|t^{m}\partial^{m}_{t}(K_{\alpha,t}(x-y)-K^{\mathcal{L}}_{\alpha, t}(x,y))\Big|,\\
\end{aligned}\right.$$
see Propositions \ref{pro1}--\ref{pro2-2}, respectively.

\item [$\bullet$]
 In Section \ref{sec-4}, as applications of the regularity estimates obtained in Section \ref{sec-3}, we use the $T1$ criterion established in \cite{ma2} to prove the boundedness on Campanato type spaces $BMO^{\gamma}_{\mathcal{L}}(\mathbb R^{n})$ of the following maximal operator and $g$-functions:
$$\left\{
\begin{aligned}
&M^{\alpha}_{\mathcal{L}}f(x):=\sup_{t>0}|e^{-t\mathcal{L}^{\alpha}}f(x)|;\\
&g^{\mathcal{L}}_{\alpha}(f)(x):=\Bigg(\int^{\infty}_{0}|D^{\mathcal{L},m}_{\alpha,t}(f)(x)|^{2}\frac{dt}{t}\Bigg)^{1/2};\\
& \widetilde{g}^{\mathcal{L}}_{\alpha}f(x):=\Bigg(\int^{\infty}_{0}|\widetilde{D}^{\mathcal{L}}_{\alpha,t}(f)(x)|^{2}\frac{dt}{t}\Bigg)^{1/2},
\end{aligned}
\right.$$
see Theorems \ref{the3}--\ref{the3.2}, respectively, where   $D^{\mathcal{L},m}_{\alpha,t}$ and $\widetilde{D}^{\mathcal{L}}_{\alpha,t}$ are the operators with the integral kernels
\begin{equation}\label{eq-1.3}
\begin{cases}
&D^{\mathcal{L},m}_{\alpha,t}(x,y):=t^{m}\partial_{t}^{m}K^{\mathcal{L}}_{\alpha,t}(x,y),\\
&\widetilde{D}^{\mathcal{L}}_{\alpha,t}(x,y):=t^{1/2\alpha}\nabla_{x}K^{\mathcal{L}}_{\alpha,t}(x,y),
\end{cases}
\end{equation}
respectively.

\end{itemize}

{\it Notations.} We will use $c$ and $C$ to denote the positive constants, which are independent of main parameters and may be different at each occurrence. By $B_{1}\sim B_{2}$, we mean that there exists a constant $C>1$ such that $C^{-1}\leq B_{1}/B_{2}\leq C.$

\section{Preliminaries}\label{sec-2}
\subsection{Schr\"odinger operators and function spaces} In this paper, let $ \delta_{0}=2- n/q$. At first, we list some properties of the potential $V$ which will be used in the sequel.
\begin{lemma}\label{lem2.4}{\rm (\cite[Lemma1.2]{she})}
\item{\rm (i)} For $0<r<R<\infty$,
$$\frac{1}{r^{n-2}}\int_{B(x,r)}V(y)dy\leq C\Big(\frac{r}{R}\Big)^{\delta}\frac{1}{R^{n-2}}\int_{B(x,R)}V(y)dy.$$
\item{\rm (ii)}
$r^{2-n}\int_{B(x,r)}V(y)dy=1$ if $r=\rho(x)$.
 $r\sim \rho(x)$ if and only if
$r^{2-n}\int_{B(x,r)}V(y)dy\sim 1.$

\end{lemma}
\begin{lemma}\label{lem2.5}{\rm (\cite[Lemma 1.4]{she})}
\item{\rm (i)} There exist $C>0$ and $k_{0}\geq1$ such that for all $x,y\in \mathbb{R}^{n}$,
$$C^{-1}\rho(x)\Big(1+|x-y|/\rho(x)\Big)^{-k_{0}}\leq \rho(y)\leq C\rho(x)\Big(1+|x-y|/\rho(x)\Big)^
{k_{0}/(1+k_{0})}.$$
In particular, $\rho(y)\sim \rho(x)$ if $|x-y|<C\rho(x)$.
\item{\rm (ii)} There exists $l_{0}>1$ such that
$$\int_{B(x,R)}\frac{V(y)}{|x-y|^{n-2}}dy\leq \frac{C}{R^{n-2}}\int_{B(x,R)}V(y)dy\leq C\Big(1+
\frac{R}{\rho(x)}\Big)^{l_{0}}.$$
\end{lemma}

\begin{lemma}\label{lem2.6}{\rm (\cite[Corollary 2.8]{DZ2})}
For every nonnegative Schwarz function $\omega$, there exist $\delta>0$ and $C>0$ such that
\begin{equation}\notag
\int_{\mathbb{R}^{n}}t^{-n/2}\omega(|x-y|/\sqrt{t})V(y)dy \leq
\begin{cases}
Ct^{-1}(\sqrt{t}/\rho(x))^{\delta}, & t< \rho(x)^{2}; \\
Ct^{-1}(\sqrt{t}/\rho(x))^{l_{0}}, & t\geq{\rho(x)^{2}},
\end{cases}
\end{equation}
where $l_{0}$ is the constant given in Lemma \ref{lem2.5}.
\end{lemma}

It is well known that  the classical Hardy space $H^{1}(\mathbb{R}^{n})$ can be defined via  the maximal function $\sup_{t>0}|e^{-t(-\Delta)}f(x)|$ (cf. \cite{15}). In this sense, we can say that the Hardy
space $H^{1}(\mathbb{R}^{n})$ is the Hardy space associated with $-\Delta$. Since 1990s, the theory of Hardy spaces associated with operators  on $\mathbb R^{n}$ has been investigated extensively. In \cite{DZ3}, Dziuba\'{n}ski and  Zienkiewicz  introduced the Hardy space $H^{1}_{\mathcal{L}}(\mathbb R^{n})$ related to Schr\"odinger operators $\mathcal{L}$  and obtained the atomic characterization and the Riesz transform characterization of $H^{1}_{\mathcal{L}}(\mathbb R^{n})$ via local Hardy spaces. By the aid of Campanato type spaces, the spaces $H^{p}_{\mathcal{L}}(\mathbb{R}^{n})(0<p\le 1)$ were introduced by Dziuba\'{n}ski and  Zienkiewicz \cite{DZ2}.  In recent years, the results of \cite{DZ3,DZ2} have been extended to other second-ordered differential operators,  and various function spaces associated to operators have been established. For further information, we refer the reader to \cite{BDY,wyx,jia,yang} and the references therein.

For a Schr\"{o}dinger operator $\mathcal{L}$, let $\{e^{-t\mathcal{L}}\}_{t>0}$ be the heat semigroup generated by $\mathcal{L}$ and denote by $K^{\mathcal{L}}_{t}(\cdot,\cdot)$ the integral kernels of $e^{-t\mathcal{L}}$. Because the potential $V\geq 0$,  the Feynman-Kac formula implies that
$$0\leq K^{\mathcal L}_{t}(x,y)\leq K_{t}(x-y):=(4\pi t)^{-n/2}e^{-\frac{|x-y|^{2}}{4t}}.$$
 The Hardy type spaces $H^{p}_{\mathcal{L}}(\mathbb{R}^{n})$, $0<p\leq 1$, are defined as follows (cf. \cite{DZ2}):
\begin{definition}
For $0<p\leq 1$, the Hardy type space $H^{p}_{\mathcal{L}}(\mathbb{R}^{n})$ is defined as  the completion of the space of compactly supported $L^{1}(\mathbb R^{n})$)-functions such that the maximal function
$$M_{\mathcal{L}}(f)(x):=\sup_{t>0}|e^{-t\mathcal{L}}(f)(x)|$$
belongs to $L^{p}(\mathbb{R}^{n})$. The quasi-norm in $H^{p}_{\mathcal{L}}(\mathbb{R}^{n})$ is defined as
$\|f\|_{H^{p}_{\mathcal{L}}}:=\|M_{\mathcal{L}}(f)\|_{L^{p}}.$
\end{definition}
Let $f$ be a locally integrable function on $\mathbb{R}^{n}$ and $B=B(x,r)$ be a ball. Denote by $f_{B}$ the mean of $f$ on $B$, i.e., $f_{B}=|B|^{-1}\int_{B}f(y)dy$. Let
\begin{equation*}
  f(B,V):=\begin{cases}
  f_{B}, &    r<\rho(x);\\
  0, &    r\geq\rho(x),
  \end{cases}
\end{equation*}
where the auxiliary function $\rho(\cdot)$  is defined as
$$\rho(x):=\sup\Bigg\{r>0:\frac{1}{r^{n-2}}\int_{B(x,y)}V(y)dy\leq 1\Bigg\}.$$

\begin{definition}\label{def1.1}
Let $0<\gamma\leq 1$. The Campanato type space $BMO^{\gamma}_{\mathcal{L}}(\mathbb{R}^{n})$ is defined as the set of all  locally integrable functions $f$ satisfying
$$\|f\|_{BMO^{\gamma}_{\mathcal{L}}}:=\sup_{B\subset\mathbb{R}^{n}}\Bigg\{\frac{1}{|B|^{1+\gamma/n}}\int_{B}|f(x)-f(B,V)|dx\Bigg\}<\infty.$$
\end{definition}
The dual space of $H^{n/(n+\gamma)}_{\mathcal{L}}(\mathbb{R}^{n})$, $0\leq\gamma< 1$, is the Campanato type space $BMO^{\gamma}_{\mathcal{L}}(\mathbb{R}^{n})$ (cf. \cite[Theorem 4.5]{ma1}).

 \subsection{The $T1$ criterion on Campanato type spaces}We denote by $L^{p}_{c}(\mathbb{R}^{n})$ the set of functions $f\in L^{p}(\mathbb{R}^{n})$, $1\leq p\leq\infty$, whose support $\text{ supp }(f)$ is a compact subset of $\mathbb{R}^{n}$.

\begin{definition}
Let $0\leq \beta<n$, $1<p\leq q<\infty$ with $1/q=1/p-\beta/n$. Let $T$ be a bounded linear operator from $L^{p}(\mathbb{R}^{n})$ into $L^{q}(\mathbb{R}^{n})$ such that
$$Tf(x)=\int_{\mathbb{R}^{n}}K(x,y)f(y)dy,\ \ \ f\in L^{p}_{c}(\mathbb{R}^{n})\ and\ a.e.\ x\notin \text{supp}(f).$$
We shall say that $T$ is a $\beta$-Schr\"{o}dinger-Calder\'{o}n-Zygmund operator with regularity exponent $\delta>0$ if there exists a constant  $C>0$ such that
\begin{itemize}
\item[(i)] $\displaystyle |K(x,y)|\leq \frac{C}{|x-y|^{n-\beta}}\Big(1+\frac{|x-y|}{\rho(x)}\Big)^{-N}$ for all $N>0$ and $x\neq y$;
\item[(ii)] $\displaystyle |K(x,y)-K(x,z)|+|K(y,x)-K(z,x)|\leq C\frac{|y-z|^{\delta}}{|x-y|^{n-\beta+\delta}}$ when $|x-y|>2|y-z|$.
\end{itemize}
\end{definition}

The following $T1$ type criterions on Campanato type spaces were established by Ma et al. \cite{ma2}.
\begin{theorem}\label{the1}{\rm (\cite[Theorem1.1]{ma2})}
Let $T$ be a $\beta$-Schr\"{o}dinger-Calder\'{o}n-Zygmund operator, $\beta\geq 0$, $0< \beta+\gamma<\min\{1,\delta\}$, with smoothness exponent $\delta$. Then $T$ is a bounded operator from $BMO^{\gamma}_{\mathcal{L}}(\mathbb R^{n})$ into $BMO^{\gamma+\beta}_{\mathcal{L}}(\mathbb R^{n})$ if and only if there exists a constant
$C$ such that
$$\Big(\frac{\rho(x)}{r}\Big)^{\gamma}\frac{1}{|B|^{1+\beta/n}}\int_{B}|T1(y)-(T1)_{B}|dy\leq C$$
for every ball $B(x,r)$, $x\in\mathbb{R}^{n}$ and $0<r\leq \rho(x)/2$.
\end{theorem}
When $\gamma=0$, the authors in \cite{ma2} also proved
\begin{theorem}\label{the2}{\rm (\cite[Theorem1.2]{ma2})}
Let $T$ be a $\beta$-Schr\"{o}dinger-Calder\'{o}n-Zygmund operator, $0\leq \beta<\min\{1,\delta\}$, with smoothness exponent $\delta$. Then $T$ is a bounded operator from $BMO_{\mathcal{L}}(\mathbb R^{n})$ into $BMO^{\beta}_{\mathcal{L}}(\mathbb R^{n})$ if and only if there exists a constant
$C$ such that
$$\log\Big(\frac{\rho(x)}{r}\Big)\frac{1}{|B|^{1+\beta/n}}\int_{B}|T1(y)-(T1)_{B}|dy\leq C$$
for every ball $B(x,r)$, $x\in\mathbb{R}^{n}$ and $0<r\leq \rho(x)/2$.
\end{theorem}
\begin{lemma}\label{lem4.1}{\rm (\cite[Remark 4.1]{ma2})}
Theorems \ref{the1}\  and  \ \ref{the2} can be also  stated in a vector-valued setting. If $Tf$ takes values in a Banach space $\mathbb{B}$ and the absolute values in the conditions are replaced by the norm in $\mathbb{B}$, then both results hold.
\end{lemma}

\section{Regularity estimates}\label{sec-3}
\subsection{Regularities of heat kernels}\label{sec-2.1} By the fundamental solution of Schr\"{o}dinger operators,  Dziuba\'{n}ski and  Zienkiewicz    proved    that the heat kernel $K^{\mathcal{L}}_{t}(\cdot,\cdot)$  satisfies the following estimates.
\begin{lemma}\label{pro2.2}
\item{\rm (i)} {\rm (\cite[Theorem 2.11]{DZ3})} For any $N>0$, there exist constants $C_{N},c>0$ such that
$$|K^{\mathcal{L}}_{t}(x,y)|\leq C_{N}t^{-n/2}e^{-c|x-y|^{2}/t}\Big(1+\frac{\sqrt{t}}{\rho(x)}+\frac{\sqrt{t}}{\rho(y)}\Big)^{-N}.$$
\item{\rm (ii)} {\rm (\cite[Theorem \ 4.11]{DZ1})} Assume that $0<\delta\leq \min\{1,\delta_{0}\}.$ For any $N>0$, there exist constants $C_{N},c>0$ such that for all $|h|<\sqrt{t}$,
$$\Big|K^{\mathcal{L}}_{t}(x+h,y)-K^{\mathcal{L}}_{t}(x,y)\Big|\leq C_{N}\Big(\frac{|h|}{\sqrt{t}}\Big)^{\delta}t^{-n/2}e^{-c|x-y|^{2}/t}\Big(1+\frac{\sqrt{t}}{\rho(x)}+
\frac{\sqrt{t}}{\rho(y)}\Big)^{-N}.$$
\end{lemma}

\begin{lemma}\label{lem1}{\rm (\cite[Proposition 2.16]{DZ2})}
 There exist constants $C,c>0$ such that for $x,y\in\mathbb{R}^{n}$ and  $t>0,$
$$\Big|K^{\mathcal{L}}_{t}(x,y)-K_{t}(x-y)\Big|\leq C \Big(\frac{\sqrt{t}}{\rho(x)}\Big)^{\delta_{0}}t^{-n/2}e^{-c|x-y|^{2}/t}.$$
\end{lemma}
In \cite{duo}, under the assumption that $V\in B_{q}, q>n$, X. Duong, L. Yan and C. Zhang obtained the following regularity estimate for the kernel $K^{\mathcal{L}}_{t}(\cdot,\cdot)$.
\begin{lemma}\label{le3.8}{\rm (\cite[Lemma 3.8]{duo})}
  Suppose that $V\in B_{q}$ for some $q>n$. For any $N>0$, there exist constants $C >0$ and $c>0$ such that
  for all $x,y\in\mathbb R^n$ and $t>0,$
\begin{eqnarray}\label{e3.11}
 | \nabla_x { K}^{\mathcal{L}}_t(x,y)|
 \leq C t^{-(n+1)/2}e^{-c{{|x-y|}^2}/{ t}}\left(1+\frac{\sqrt{t}}{\rho(x)}
+\frac{\sqrt{t}}{\rho(y)}\right)^{-N}.
\end{eqnarray}
\end{lemma}
By the perturbation theory for semigroups of operators,
\begin{eqnarray}\label{p-2}
  K_{t}(x-y)-K_{t}^{\mathcal{L}}(x,y) &=& \int^{t}_{0}\int_{\mathbb{R}^{n}}K_{t-s}(w-x)V(w)K^{\mathcal{L}}_{s}(w,y)dwds \\
   &=& \int^{t/2}_{0}\int_{\mathbb{R}^{n}}K_{t-s}(w-x)V(w)K^{\mathcal{L}}_{s}(w,y)dwds\nonumber \\
   &&+ \int^{t/2}_{0}\int_{\mathbb{R}^{n}}K_{s}(w-x)V(w)K^{\mathcal{L}}_{t-s}(w,y)dwds.\nonumber
\end{eqnarray}

Similar to \cite[Proposition 2.16]{DZ2}, we can prove the following lemma.
\begin{lemma}\label{lem1-1}
Suppose that $V\in B_{q}$ for some $q>n$. There exist constants $C,c>0$ such that
$$\Big|\nabla_{x}K_{t}(x-y)-\nabla_{x}K^{\mathcal{L}}_{t}(x,y)\Big|\leq Ct^{-(n+1)/2}e^{-c|x-y|^{2}/t}\min\Bigg\{\Big(\frac{\sqrt{t}}{\rho(x)}\Big)^{\delta_{0}},\
\Big(\frac{\sqrt{t}}{\rho(y)}\Big)^{\delta_{0}}\Bigg\}.$$
\end{lemma}
\begin{proof}
If $t\geq\rho(y)^{2}$, it is easy to see that
$$\Big|\nabla_{x}K_{t}(x-y)-\nabla_{x}K^{\mathcal{L}}_{t}(x,y)\Big|\leq Ct^{-(n+1)/2}e^{-c|x-y|^{2}/t}\Big(\frac{\sqrt{t}}{\rho(y)}\Big)^{\delta_{0}}.$$
If $t<\rho(y)^{2}$, by (\ref{p-2}), we get
\begin{eqnarray}\notag
  \Big|\nabla_{x}K_{t}(x-y)-\nabla_{x}K^{\mathcal{L}}_{t}(x,y)\Big| &=& \Big|\int^{t}_{0}\int_{\mathbb{R}^{n}}\nabla_{x}K_{t-s}(x-w)V(w)K^{\mathcal{L}}_{s}(w,y)dwds\Big| \\ \label{12}
   &\leq&I_{1}+I_{2},
\end{eqnarray}
where
\begin{equation*}
\left\{ \begin{aligned}
 I_{1}&:= \int^{t/2}_{0}\int_{\mathbb{R}^{n}}|\nabla_{x}K_{t-s}(x-w)|V(w)K^{\mathcal{L}}_{s}(w,y)dwds;\\
 I_{2}&:=\int^{t/2}_{0}\int_{\mathbb{R}^{n}}|\nabla_{x}K_{s}(x-w)|V(w)K^{\mathcal{L}}_{t-s}(w,y)dwds.
 \end{aligned}\right.
\end{equation*}

For $I_{1}$, using Lemmas \ref{le3.8} and \ref{lem2.6}, we get
\begin{eqnarray*}
  I_{1} &=&\int^{t/2}_{0}\int_{|w-y|<|x-y|/2}|\nabla_{x}K_{t-s}(x-w)|V(w)K^{\mathcal{L}}_{s}(w,y)dwds\\
  &&+\int^{t/2}_{0}\int_{|w-y|\geq|x-y|/2}|\nabla_{x}K_{t-s}(x-w)|V(w)K^{\mathcal{L}}_{s}(w,y)dwds  \\
   &\leq& Ct^{-(n+1)/2}e^{-c|x-y|^{2}/t}\int^{t/2}_{0}\int_{|w-y|<|x-y|/2}V(w)s^{-n/2}e^{-c|w-y|^{2}/s}dwds\\
   &&+ Ct^{-(n+1)/2}\int^{t/2}_{0}\int_{|w-y|\geq|x-y|/2}V(w)s^{-n/2}e^{-c(|w-y|^{2}+|x-y|^{2})/s}dwds \\
   &\leq& Ct^{-(n+1)/2}e^{-c|x-y|^{2}/t}\int^{t/2}_{0}\frac{1}{s}\Big(\frac{\sqrt{s}}{\rho(y)}\Big)^{\delta_{0}}ds \\
   &=&  Ct^{-(n+1)/2}e^{-c|x-y|^{2}/t}\Big(\frac{\sqrt{t}}{\rho(y)}\Big)^{\delta_{0}}.
\end{eqnarray*}
Similar to $I_{1}$, for the term $I_{2}$, we can obtain
$$I_{2}\leq Ct^{-(n+1)/2}e^{-c|x-y|^{2}/t}\Big(\frac{\sqrt{t}}{\rho(x)}\Big)^{\delta_{0}}.$$
It follows from (i) of Lemma \ref{lem2.5} that
$$\frac{\sqrt{t}}{\rho(x)}\leq C\Bigg(1+\frac{|x-y|}{\sqrt{t}}\frac{\sqrt{t}}{\rho(y)}\Bigg)^{l_{0}}\frac{\sqrt{t}}{\rho(y)}\leq
C_{\varepsilon}e^{\varepsilon|x-y|^{2}/t}\frac{\sqrt{t}}{\rho(y)},$$
where $\varepsilon>0$ is an arbitrary small constant. Hence
$$I_{2}\leq Ct^{-(n+1)/2}e^{-c|x-y|^{2}/t}\Big(\frac{\sqrt{t}}{\rho(y)}\Big)^{\delta_{0}}.$$

\end{proof}

\begin{lemma}\label{lem2}{\rm (\cite[Proposition 2.17]{DZ2})}
Let $0<\delta<\min\{1,\delta_{0}\}$. For $|y-z|<\min\{C\rho(x),|x-y|/4\},$ there exist constants $C,c>0$ such that
$$\Big|\Big(K^{\mathcal{L}}_{t}(x,y)-K_{t}(x-y)\Big)-\Big(K^{\mathcal{L}}_{t}(x,z)-K_{t}(x-z)\Big)\Big|\leq C\Big(\frac{|y-z|}{\rho(y)}\Big)^{\delta}t^{-n/2}e^{-c|x-y|^{2}/t}.$$
\end{lemma}

\begin{lemma}\label{lem2-2}
Suppose that $V\in B_{q}$ for some $q>n$. Let $\delta_{1}=1-n/q$. If $|u|\leq \min\{|x-y|/4, C\rho(x)\}$, there exist constants $C,c>0$ such that for $0<\delta'<\delta_{1}$,
\begin{multline*}
 \Big|\Big(\nabla_{x}K^{\mathcal{L}}_{t}(x,y)-\nabla_{x}K_{t}(x-y)\Big)-\Big(\nabla_{x}K^{\mathcal{L}}_{t}(x+u,y)-\nabla_{x}K_{t}(x+u-y)\Big)\Big|\\
  \quad\leq Ct^{-(n+1)/2}e^{-c|x-y|^{2}/t}\Big(\frac{|u|}{\rho(y)}\Big)^{\delta'}.
\end{multline*}
\end{lemma}
\begin{proof}
We prove this lemma by the same argument as Lemma \ref{lem1-1}. It is enough to verify that
\begin{eqnarray}\label{13}
  &&\Big|\Big(\nabla_{x}K^{\mathcal{L}}_{t}(x,y)-\nabla_{x}K_{t}(x-y)\Big)-\Big(\nabla_{x}K^{\mathcal{L}}_{t}(x+u,y)-\nabla_{x}K_{t}(x+u-y)\Big)\Big|\\
  &&\quad\leq C_{\epsilon}t^{-(n+1)/2}e^{\epsilon|x-y|^{2}/t}\Big(\frac{|u|}{\rho(y)}\Big)^{\delta_{1}},\nonumber
\end{eqnarray}
where $\epsilon>0$ is an arbitrary small constant. In fact, under the condition $|u|<|x-y|/4$, it is easy to see that $|x-y|\sim |x+u-y|$. We can deduce from Lemma \ref{le3.8} that
\begin{eqnarray}\label{eq-10}
&&\Big|\Big(\nabla_{x}K^{\mathcal{L}}_{t}(x,y)-\nabla_{x}K_{t}(x-y)\Big)-\Big(\nabla_{x}K^{\mathcal{L}}_{t}(x+u,y)-\nabla_{x}K_{t}(x+u-y)\Big)\Big|\\
&&\quad\leq \Big|\nabla_{x}K^{\mathcal{L}}_{t}(x,y)\Big|+\Big|\nabla_{x}K_{t}(x-y)\Big|+\Big|\nabla_{x}K^{\mathcal{L}}_{t}(x+u,y)\Big|+\Big|\nabla_{x}K_{t}(x+u-y)\Big|\nonumber\\
&&\quad\leq C t^{-(n+1)/2}e^{-c{{|x-y|}^2}/t}.\nonumber
\end{eqnarray}
Then, for $\delta'\in(0,\delta_{1})$, it follows from (\ref{13})\ and  (\ref{eq-10}) that
\begin{eqnarray*}
&&\Big|\Big(\nabla_{x}K^{\mathcal{L}}_{t}(x,y)-\nabla_{x}K_{t}(x-y)\Big)-\Big(\nabla_{x}K^{\mathcal{L}}_{t}(x+u,y)-\nabla_{x}K_{t}(x+u-y)\Big)\Big|\\
&&\quad\leq C \Bigg\{t^{-(n+1)/2}e^{\epsilon|x-y|^{2}/t}\Big(\frac{|u|}{\rho(y)}\Big)^{\delta_{1}}\Bigg\}^{\delta'/\delta_{1}}
\Bigg\{t^{-(n+1)/2}e^{-c{{|x-y|}^2}/t}\Bigg\}^{1-\delta'/\delta_{1}},
\end{eqnarray*}
 which  gives the desired estimate.

 Now we prove (\ref{13}). Since the case for $|u|\geq\rho(y)$ is trivial, so we may assume $|u|<\rho(y)$. If $t\leq 2|u|^{2}$, the required estimate follows from Lemma \ref{lem1-1}. Hence we  consider the case $t>2|u|^{2}$ only. Recall that for the classical heat kernel $K_{t}(\cdot)$, it holds
$$\Big|\nabla^{2}_{x}K_{t}(x)\Big|\leq Ct^{-n/2-1}e^{-c|x|^{2}/t}.$$
A direct computation gives
\begin{equation}\label{14}
  \Big|\nabla_{x}K_{t}(x+u)-\nabla_{x}K_{t}(x)\Big|\leq C|u|t^{-n/2-1},
\end{equation}
and for $|u|\leq |x|/2$,
\begin{equation}\label{15}
  \Big|\nabla_{x}K_{t}(x+u)-\nabla_{x}K_{t}(x)\Big|\leq C|u|t^{-n/2-1}e^{-c|x|^{2}/t}.
\end{equation}
Similar to (\ref{12}), we split $$\Big|\Big(\nabla_{x}K^{\mathcal{L}}_{t}(x,y)-\nabla_{x}K_{t}(x-y)\Big)-\Big(\nabla_{x}K^{\mathcal{L}}_{t}(x+u,y)-\nabla_{x}K_{t}(x+u-y)\Big)\Big|\leq J_{1}+J_{2},$$ where
\begin{equation*}
\left\{  \begin{aligned}
 J_{1}&:=\int^{t/2}_{0}\int_{\mathbb{R}^{n}}\Big|\nabla_{x}K_{t-s}(w-(x+u))-\nabla_{x}K_{t-s}(w-x)\Big|V(w)K^{\mathcal{L}}_{s}(w,y)dwds;\\
 J_{2}&:=\int^{t/2}_{0}\int_{\mathbb{R}^{n}}\Big|\nabla_{x}K_{s}(w-(x+u))-\nabla_{x}K_{s}(w-x)\Big|V(w)K^{\mathcal{L}}_{t-s}(w,y)dwds.
  \end{aligned}\right.
\end{equation*}
For $J_{1}$, if $t<2\rho(y)^{2}$, using Lemma \ref{lem2.6} and (\ref{14}), we get
\begin{eqnarray*}
  J_{1} &\leq& C|u|t^{-n/2-1}\int^{t/2}_{0}\int_{\mathbb{R}^{n}}V(w)s^{-n/2}e^{-c|y-w|^{2}/s}dwds  \\
   &\leq& C|u|t^{-n/2-1}\Big(\frac{\sqrt{t}}{\rho(y)}\Big)^{\delta_{0}} \\
   &\leq& Ct^{-(n+1)/2}\Big(\frac{|u|}{\rho(y)}\Big)^{\delta_{1}}.
\end{eqnarray*}
If $t\geq 2\rho(y)^{2}$, applying Lemmas \ref{lem2.6} and  \ref{pro2.2}, we have
\begin{eqnarray*}
  J_{1} &\leq&C|u|t^{-n/2-1}\int^{t/2}_{0}\int_{\mathbb{R}^{n}}V(w)s^{-n/2}e^{-c|y-w|^{2}/s}\Big(1+\frac{\sqrt{s}}{\rho(y)}\Big)^{-N}dwds  \\
   &\leq&C|u|t^{-n/2-1}\Bigg\{\int^{\rho(y)^{2}}_{0}\frac{1}{s}\Big(\frac{\sqrt{s}}{\rho(y)}\Big)^{\delta_{0}}ds+\int^{t/2}_{\rho(y)^{2}}
   \frac{1}{s}\Big(\frac{\sqrt{s}}{\rho(y)}\Big)^{l_{0}-N}ds\Bigg\} \\
   &\leq&C|u|t^{-n/2-1}  \\
   &\leq& Ct^{-(n+1)/2}\Big(\frac{|u|}{\rho(y)}\Big)^{\delta_{1}},
\end{eqnarray*}
where $N$ is chosen large enough satisfying $N>l_{0}$.

To estimate $J_{2}$, we use Lemma \ref{pro2.2} and write $J_{2}\leq C(J_{2,1}+J_{2,2}+J_{2,3})$, where
\begin{equation*}
\left\{  \begin{aligned}
 J_{2,1}&:=t^{-n/2}\Big(1+\frac{\sqrt{t}}{\rho(y)}\Big)^{-N}\int^{|u|^{2}}_{0}\int_{\mathbb{R}^{n}}\Big|\nabla_{x}K_{s}(w-(x+u))-\nabla_{x}K_{s}
(w-x)\Big|V(w)dwds;\\
J_{2,2}&:=t^{-n/2}\Big(1+\frac{\sqrt{t}}{\rho(y)}\Big)^{-N}\int^{t/2}_{|u|^{2}}\int_{|w-x|<2|u|}\Big|\nabla_{x}K_{s}(w-(x+u))-\nabla_{x}K_{s}
(w-x)\Big|V(w)dwds;\\
J_{2,3}&:=t^{-n/2}\Big(1+\frac{\sqrt{t}}{\rho(y)}\Big)^{-N}\int^{t/2}_{|u|^{2}}\int_{|w-x|\geq2|u|}\Big|\nabla_{x}K_{s}(w-(x+u))-\nabla_{x}K_{s}
(w-x)\Big|V(w)dwds.
\end{aligned}\right.
\end{equation*}
Notice that $\rho(x+u)\sim\rho(x)$ as $|u|\leq \rho(x)$. It holds
\begin{eqnarray*}
  J_{2,1} &\leq& Ct^{-n/2}\Big(1+\frac{\sqrt{t}}{\rho(y)}\Big)^{-N}\int^{|u|^{2}}_{0}\frac{1}{s^{3/2}}\Big(\frac{\sqrt{s}}{\rho(x)}\Big)^{\delta_{0}}ds  \\
   &\leq&Ct^{-n/2}|u|^{-1}\Big(\frac{|u|}{\rho(y)}\Big)^{-1}\Big(\frac{\sqrt{t}}{|u|}\Big)^{-1}\Big(1+\frac{\sqrt{t}}{\rho(y)}\Big)^{-N+1}
   \Big(\frac{|u|}{\rho(y)}\Big)^{\delta_{0}}\Big(\frac{\rho(y)}{\rho(x)}\Big)^{\delta_{0}}  \\
   &\leq& Ct^{-(n+1)/2}\Big(1+\frac{\sqrt{t}}{\rho(y)}\Big)^{-N+1}\Big(\frac{|u|}{\rho(y)}\Big)^{\delta_{1}}\Big(\frac{\rho(y)}{\rho(x)}\Big)^{\delta_{0}},
\end{eqnarray*}
where in the last inequality we have used the fact that $\delta_{0}=2-n/q$. By Lemmas \ref{lem2.4}\ and \ \ref{lem2.5}, we apply (\ref{14}) to get
\begin{eqnarray*}
  J_{2,2} &\leq& Ct^{-n/2}\Big(1+\frac{\sqrt{t}}{\rho(y)}\Big)^{-N}\int^{t/2}_{|u|^{2}}\int_{|w-x|<2|u|}|u|s^{-n/2-1}V(w)dwds \\
   &\leq& Ct^{-n/2}\Big(1+\frac{\sqrt{t}}{\rho(y)}\Big)^{-N}\int^{t/2}_{|u|^{2}}|u|^{n-1}s^{-n/2-1}\Big(\frac{|u|}{\rho(x)}\Big)^{\delta_{0}}ds \\
   &\leq& Ct^{-n/2}|u|^{-1}\Big(\frac{|u|}{\rho(y)}\Big)^{\delta_{0}}\Big(1+\frac{\sqrt{t}}{\rho(y)}\Big)^{-N}\Big(\frac{\rho(y)}{\rho(x)}\Big)^{\delta_{0}} \\
   &\leq& Ct^{-(n+1)/2}\Big(1+\frac{\sqrt{t}}{\rho(y)}\Big)^{-N+1}\Big(\frac{|u|}{\rho(y)}\Big)^{\delta_{1}}\Big(\frac{\rho(y)}{\rho(x)}\Big)^{\delta_{0}}.
\end{eqnarray*}
For $J_{2,3}$, if $t\leq 2\rho(x)^{2}$, it can be deduced from Lemma \ref{lem2.6} and (\ref{15}) that
\begin{eqnarray*}
  J_{2,3} &\leq&Ct^{-n/2}\Big(1+\frac{\sqrt{t}}{\rho(y)}\Big)^{-N}\int^{t/2}_{|u|^{2}}\int_{|w-x|\geq 2|u|}|u|s^{-n/2-1}e^{-c|w-x|^{2}/s}V(w)dwds  \\
   &\leq&Ct^{-n/2}\Big(1+\frac{\sqrt{t}}{\rho(y)}\Big)^{-N}|u|\int^{t/2}_{|u|^{2}}\frac{1}{s^{2}}\Big(\frac{\sqrt{s}}{\rho(x)}\Big)^{\delta_{0}}ds  \\
   &\leq&Ct^{-n/2}\Big(1+\frac{\sqrt{t}}{\rho(y)}\Big)^{-N}|u|^{-1}\Big(\frac{|u|}{\rho(y)}\Big)^{\delta_{0}}
   \Big(\frac{\rho(y)}{\rho(x)}\Big)^{\delta_{0}}  \\
   &\leq& Ct^{-(n+1)/2}\Big(1+\frac{\sqrt{t}}{\rho(y)}\Big)^{-N+1}\Big(\frac{|u|}{\rho(y)}\Big)^{\delta_{1}}\Big(\frac{\rho(y)}{\rho(x)}\Big)^{\delta_{0}}.
\end{eqnarray*}
If $t>2\rho(x)^{2}$, then
\begin{eqnarray}\notag
  J_{2,3} &\leq& Ct^{-n/2}\Big(1+\frac{\sqrt{t}}{\rho(y)}\Big)^{-N}\int^{\rho(x)^{2}}_{|u|^{2}}\int_{|w-x|\geq 2|u|}\Big|\nabla_{x}K_{s}(w-(x+u))-\nabla_{x}K_{s}(w-x)\Big|V(w)dwds  \\ \notag
  &&+ Ct^{-n/2}\Big(1+\frac{\sqrt{t}}{\rho(y)}\Big)^{-N}\int^{t/2}_{\rho(x)^{2}}\int_{|w-x|\geq 2|u|}|u|s^{-n/2-1}e^{-c|w-x|^{2}/s}V(w)dwds\\ \notag
   &\leq& Ct^{-(n+1)/2}\Big(\frac{|u|}{\rho(y)}\Big)^{\delta_{1}}\Big(1+\frac{\sqrt{t}}{\rho(y)}\Big)^{-N+1}\Big(\frac{\rho(y)}{\rho(x)}\Big)^{\delta_{0}}\\
   \label{e-1}
   &&+Ct^{-n/2}\Big(1+\frac{\sqrt{t}}{\rho(y)}\Big)^{-N}\int^{t/2}_{\rho(x)^{2}}\frac{|u|}{s}\frac{1}{s}\Big(\frac{\sqrt{s}}{\rho(x)}\Big)^{l_{0}}
   ds  \\ \notag
   &\leq& Ct^{-(n+1)/2}\Big(\frac{|u|}{\rho(y)}\Big)^{\delta_{1}}\Big(1+\frac{\sqrt{t}}{\rho(y)}\Big)^{-N+1}\Big(\frac{\rho(y)}{\rho(x)}\Big)^{\delta_{0}}  \\ \notag
   &&+Ct^{-(n+1)/2}\frac{|u|}{\rho(y)}\Big(1+\frac{\sqrt{t}}{\rho(y)}\Big)^{-N-1+l_{0}}\Big(\frac{\rho(y)}{\rho(x)}\Big)^{l_{0}},
\end{eqnarray}
where in (\ref{e-1}) we have used the estimate obtained for $t\leq 2\rho(x)^{2}$ and Lemma \ref{lem2.6} for $s\geq \rho(x)^{2}$.

By Lemma \ref{lem2.5},
$$\frac{\rho(y)}{\rho(x)}\leq C\Big(1+\frac{|x-y|}{\sqrt{t}}\frac{\sqrt{t}}{\rho(y)}\Big)^{m_{0}}\leq C_{\epsilon}e^{\epsilon|x-y|^{2}/t}
\Big(1+\frac{\sqrt{t}}{\rho(y)}\Big)^{m_{0}},$$
where $\epsilon>0$ is an arbitrary small constant. Choosing $N$ large enough in the estimates of $J_{2,1},J_{2,2}$ and $J_{2,3}$, we obtain (\ref{13}) and hence Lemma \ref{lem2-2} is proved.

\end{proof}

 We can obtain the following estimates, which generalize \cite[Lemmas 3.7\ and \ 3.8]{Zhang}. We also refer to \cite[(57)]{LL} for the case $m=1$ in the setting of Heisenberg groups.
\begin{proposition}\label{pro2.5}
\item{\rm (i)} There exist constants $C,c>0$ such that
$$\Big|t^{m}\partial_{t}^{m}K^{\mathcal{L}}_{t}(x,y)-t^{m}\partial_{t}^{m}K_{t}(x-y)\Big|\leq Ct^{-n/2}e^{-c|x-y|^{2}/t}\min\Bigg\{\Big(\frac{\sqrt{t}}{\rho(x)}\Big)^{\delta_{0}},
\frac{\sqrt{t}}{\rho(y)}\Big)^{\delta_{0}}\Bigg\}.$$
\item{\rm (ii)} For every $0<\delta<\min\{1,\delta_{0}\}$ and $|y-z|<\min\{\rho(x),|x-y|/4\}$, there exist constants $C,c>0$ such that
    \begin{multline*}
    \Bigg|\Big(t^{m}\partial_{t}^{m}K^{\mathcal{L}}_{t}(x,y)-t^{m}\partial_{t}^{m}K_{t}(x-y)\Big)
-\Big(t^{m}\partial_{t}^{m}K^{\mathcal{L}}_{t}(x,z)-t^{m}\partial_{t}^{m}K_{t}(x-z)\Big)\Bigg|\\ \leq C\Big(\frac{|y-z|}{\rho(y)}\Big)^{\delta}t^{-n/2}e^{-c|x-y|^{2}/t}.
\end{multline*}
\end{proposition}
\begin{proof}For $t>0$ and $m\in\mathbb{Z}_{+}$, define
$Q^{\mathcal{L}}_{t,m}(x,y):= t^{m}\partial_{t}^{m}K^{\mathcal{L}}_{t}(x,y)$
and $Q_{t,m}(x-y):= t^{m}\partial_{t}^{m}K_{t}(x-y).$
The proof of (i) is similar to \cite[Lemmas 4.10\ and \ 4.11]{wyx}, so we omit the details.

For (ii), by (\ref{p-2}), we get
\begin{eqnarray*}
 &&\Big(K_{t}(x+u-y)-K^{\mathcal{L}}_{t}(x+u,y)\Big)-\Big(K_{t}(x-y)-K_{t}^{\mathcal{L}}(x,y)\Big)\\  &&= \int^{t/2}_{0}\int_{\mathbb{R}^{n}}\Big(K_{t-s}(w-(x+u))-K_{t-s}(w-x)\Big)V(w)K^{\mathcal{L}}_{s}(w,y)dwds\\
 &&\quad + \int^{t/2}_{0}\int_{\mathbb{R}^{n}}\Big(K_{s}(w-(x+u))-K_{s}(w-x)\Big)V(w)K^{\mathcal{L}}_{t-s}(w,y)dwds.
\end{eqnarray*}
Similar to \cite[Proposition 4.8]{wyx}, we use a direct calculus to deduce that
\item{\rm (1)} If $m$ is even with $m\geq 2$, there exists a sequence of coefficients $\{C_{m,j}\}_{m\geq 2,2\leq j\leq m/2}$ such that
\begin{eqnarray}\label{eq-2.2}
  &&t^{m}\frac{d^{m}}{dt^{m}}\Big\{\Big(K_{t}(x+u-y)-K^{\mathcal{L}}_{t}(x+u,y)\Big)-\Big(K_{t}(x-y)-K_{t}^{\mathcal{L}}(x,y)\Big)\Big\}\nonumber\\
  &&= \frac{m+1}{2}(E_{1}+E_{2})
  +\sum^{m/2}_{j=2}(C_{m-1,j-1}+C_{m-1,j})\Big(E^{j}_{3,1}+E_{3,2}^{j}\Big)
   +E_{4}+E_{5}.
\end{eqnarray}
\item{\rm (2)} If $m$ is odd with $m\geq 3$, there exists a sequence of coefficients $\{C_{m,j}\}_{m\geq 3,2\leq j\leq [m/2]}$ such that
\begin{eqnarray}\label{eq-2.3}
  &&t^{m}\frac{d^{m}}{dt^{m}}\Big\{\Big(K_{t}(x+u-y)-K^{\mathcal{L}}_{t}(x+u,y)\Big)-\Big(K_{t}(x-y)-K_{t}^{\mathcal{L}}(x,y)\Big)\Big\}\nonumber\\
   &&=\frac{m+1}{2}(E_{1}+E_{2})+ E_{4}-E_{5}
   +\sum^{[m/2]}_{j=2}(C_{m-1,j-1}+C_{m-1,j})\Big(E^{j}_{3,1}+E^{j}_{3,2}\Big)\nonumber\\
   &&+2C_{m,[m/2]}\frac{d^{[m/2]}}{dt^{[m/2]}}\Big(K_{t/2}(w-(x+u))-K_{t/2}(w-x)\Big)V(w)\frac{d^{[m/2]}}{dt^{[m/2]}}K^{\mathcal{L}}_{t/2}(w,y).
  \end{eqnarray}
Here in the above  (\ref{eq-2.2})\ and \ (\ref{eq-2.3}),
\begin{equation*}
\left\{  \begin{aligned}
E_{1}:&=\int_{\mathbb{R}^{n}}t\Big\{t^{m-1}\frac{d^{m-1}}{dt^{m-1}}\Big(K_{t/2}(w-x-u)-K_{t/2}(w-x)\Big)\Big\}V(w)K^{\mathcal{L}}_{t/2}(w,y)dw;\\
E_{2}:&=\int_{\mathbb{R}^{n}}t\Big(K_{t/2}(w-x-u))-K_{t/2}(w-x)\Big)V(w)\Big(t^{m-1}\frac{d^{m-1}}{dt^{m-1}}K^{\mathcal{L}}_{t/2}(w,y)\Big)dw;\\
E_{3,1}^{j}:&=\int_{\mathbb{R}^{n}}
t\Big\{t^{m-j}\frac{d^{m-j}}{dt^{m-j}}\Big(K_{t/2}(w-x-u)-K_{t/2}(w-x)\Big)\Big\}V(w)\Big(t^{j-1}\frac{d^{j-1}}{dt^{j-1}}K^{\mathcal{L}}_{t/2}(w,y)\Big)dw;\\
E_{3,2}^{j}:&=\int_{\mathbb{R}^{n}}
t\Big\{t^{j-1}\frac{d^{j-1}}{dt^{j-1}}\Big(K_{t/2}(w-x-u)-K_{t/2}(w-x)\Big)\Big\}V(w)\Big(t^{m-j}\frac{d^{m-j}}{dt^{m-j}}K^{\mathcal{L}}_{t/2}(w,y)\Big)dw;\\
E_{4}:&=\int^{t/2}_{0}\int_{\mathbb{R}^{n}}\Big(t^{m}\frac{d^{m}}{dt^{m}}(K_{t-s}(w-x-u)-K_{t-s}(w-x)\Big)V(w)K^{\mathcal{L}}_{s}(w,y)dwds;\\
E_{5}:&=\int^{t/2}_{0}\int_{\mathbb{R}^{n}}\Big(K_{s}(w-x-u)-K_{s}(w,x)\Big)V(w)\Big(t^{m}\frac{d^{m}}{dt^{m}}K^{\mathcal{L}}_{t-s}(w,y)\Big)dwds.
\end{aligned}\right.
\end{equation*}

Below, for the sake of simplicity, we only estimate $E_{1}, E_{4}, E_{5}$. The estimations for $E_{2}$, $E_{3,1}^{j}$, $E^{j}_{3,2}$ are similar, and so we omit the details. By the mean value theorem, we know that there exist constants $C, c$ such that
\begin{equation}\label{24}
\Big|Q_{t,m}(x+u)-Q_{t,m}(x)\Big|\leq
\left\{\begin{aligned}
&C|u|t^{-(n+1)/2},&\quad \forall\   x, u \in\mathbb R^{n},\  t\in(0,\infty);\\
&C|u|t^{-(n+1)/2}e^{-c|x|^{2}/t},&\quad |u|\leq |x|/2,\  t>0.
\end{aligned}\right.
\end{equation}
We divide $E_{1}$ as $E_{1}\leq E_{1,1}+E_{1,2}$, where
\begin{equation*}
 \left\{ \begin{aligned}
 E_{1,1}&:=t\int_{|w-x|<2u}\Big|Q_{t/2,m-1}(w-(x+u))-Q_{t/2,m-1}(w-x)\Big|V(w)K^{\mathcal{L}}_{t/2}(w,y)dw;\\
 E_{1,2}&:=t\int_{|w-x|\geq2u}\Big|Q_{t/2,m-1}(w-(x+u))-Q_{t/2,m-1}(w-x)\Big|V(w)K^{\mathcal{L}}_{t/2}(w,y)dw.
  \end{aligned}\right.
\end{equation*}
If $t<2\rho(y)^{2}$, for $E_{1,1}$, By (\ref{24}), Lemma \ref{lem2.6} (i) and Lemma \ref{pro2.2}, we obtain
\begin{eqnarray*}
  E_{1,1} &\leq& Ct^{1-n/2}\int_{|w-x|<2|u|}|u|t^{-(n+1)/2}V(w)e^{-c|w-y|^{2}/t}dw  \\
   &\leq& Ct^{1-n/2}|u|\frac{1}{t^{3/2}}\Big(\frac{\sqrt{t}}{\rho(y)}\Big)^{\delta_{0}} \\
   &=&Ct^{-n/2}\frac{|u|}{\sqrt{t}}\Big(\frac{\sqrt{t}}{\rho(y)}\Big)^{\delta_{0}}\leq Ct^{-n/2}\Big(\frac{|u|}{\rho(y)}\Big)^{\delta}.
    \end{eqnarray*}
Similar to $E_{1,1}$, by (\ref{24}) again, we get
\begin{eqnarray*}
  E_{1,2} &\leq&Ct^{1-n/2}\int_{|w-x|\geq 2|u|}|u|t^{-(n+1)/2}e^{-c|w-y|^{2}/t}V(w)dw  \\
   &\leq&Ct^{1-n/2}\frac{|u|}{t^{3/2}}\Big(\frac{\sqrt{t}}{\rho(y)}\Big)^{\delta_{0}}  \\
   &\leq&Ct^{-n/2}\Big(\frac{|u|}{\rho(y)}\Big)^{\delta}.
\end{eqnarray*}
If $t>2\rho(y)^{2}$, for $E_{1,1}$, by (\ref{24}), Lemma \ref{lem2.6} (ii) and Lemma \ref{pro2.2}, we obtain
\begin{eqnarray*}
  E_{1,1} &\leq& Ct^{1-n/2}\Big(1+\frac{\sqrt{t}}{\rho(y)}\Big)^{-N}\int_{|w-x|<2|u|}|u|t^{-(n+1)/2}V(w)e^{-c|w-y|^{2}/t}dw  \\
   &\leq& Ct^{1-n/2}\Big(1+\frac{\sqrt{t}}{\rho(y)}\Big)^{-N}|u|\frac{1}{t^{3/2}}\Big(\frac{\sqrt{t}}{\rho(y)}\Big)^{l_{0}} \\
   &\leq& Ct^{-n/2}\Big(\frac{|u|}{\rho(y)}\Big)^{\delta}\Big(1+\frac{\sqrt{t}}{\rho(y)}\Big)^{-N+l_{0}}.
   \end{eqnarray*}
Similar to $E_{1,1}$, we can also choose $N$ large enough such that
$$E_{1,2}\leq Ct^{-n/2}\Big(\frac{|u|}{\rho(y)}\Big)^{\delta}\Big(1+\frac{\sqrt{t}}{\rho(y)}\Big)^{-N+l_{0}}.$$

For $E_{4}$ and $E_{5}$, similar to \cite[Lemma 10]{LL}, by Lemma \ref{pro2.2} and  (\ref{24}), we can obtain
\begin{equation}\label{23}
  |E_{4}+E_{5}|\leq C_{\epsilon}t^{-n/2}e^{\epsilon|x-y|^{2}/t}\Big(\frac{|u|}{\rho(y)}\Big)^{\delta},
\end{equation}
where $\epsilon>0$ is an arbitrary small constant. Hence Proposition \ref{pro2.5} is proved.
\end{proof}

\subsection{Fractional heat kernels associated with $\mathcal{L}$}\label{sec-2.2}

In the following, we will derive some regularity estimates for the fractional heat kernels related with $\mathcal{L}$.  For $\alpha\in(0,1)$, the fractional power of $\mathcal{L}$, denoted by $\mathcal{L}^{\alpha}$, is defined as
\begin{equation}\label{7}
  \mathcal{L}^{\alpha}:=\frac{1}{\Gamma(-\alpha)}\int^{\infty}_{0}\Big(e^{-t\sqrt{\mathcal{L}}}f(x)-f(x)\Big)\frac{dt}{t^{1+2\alpha}}\quad   \forall\ f\in L^{2}(\mathbb{R}^{n}).
\end{equation}
We use the subordinative formula to express the integral kernel $K^{\mathcal{L}}_{\alpha,t}(\cdot,\cdot)$ of $e^{-t\mathcal{L}^{\alpha}}$ as (cf. \cite{gri})
\begin{equation*}
  K^{\mathcal{L}}_{\alpha,t}(x,y)=\int^{\infty}_{0}\eta^{\alpha}_{t}(s)K^{\mathcal{L}}_{s}(x,y)ds,
\end{equation*}
where $\eta^{\alpha}_{t}(\cdot)$ satisfies
\begin{equation}\label{eq-2.1}
 \left\{\begin{aligned}
 &\eta^{\alpha}_{t}(s)=1/t^{1/\alpha}\eta_{1}^{\alpha}(s/t^{1/\alpha});\\
 &\eta^{\alpha}_{t}(s)\leq t/s^{1+\alpha}\quad   \forall\  s,t>0;\\
 &\int^{\infty}_{0}s^{-r}\eta^{\alpha}_{1}(s)ds<\infty, \ \ r>0;\\
 &\eta^{\alpha}_{t}(s)\simeq t/s^{1+\alpha}\ \ \forall\  s\geq t^{1/\alpha}>0.
 \end{aligned}\right.
\end{equation}
By the subordinative formula (\ref{p-1}) and Lemma \ref{pro2.2}, Li et al. in \cite{Li2} proved the following estimates for $K^{\mathcal{L}}_{\alpha,t}(\cdot,\cdot)$.
\begin{proposition}\label{pro2.51}{\rm \cite[Propositions 3.1\ and \ 3.2]{Li2}}
Let $0<\alpha<1$.
\item{\rm (i)} For any $N>0$, there exists a constant $C_{N}>0$ such that
$$\Big|K^{\mathcal L}_{\alpha,t}(x,y)\Big|\leq \frac{C_{N}t}{(\sqrt{t^{1/\alpha}}+|x-y|)^{n+2\alpha}}\Big(1+\frac{\sqrt{t^{1/\alpha}}}{\rho(x)}+\frac{\sqrt{t^{1/\alpha}}}{\rho(y)}\Big)^{-N}.$$
\item{\rm (ii)} Let $0<\delta\leq  \min\{1,\delta_{0}\}$. For any $N>0$, there exists a constant $C_{N}>0$ such that for all $|h|\leq \sqrt{t^{1/\alpha}}$,
$$\Big|K^{\mathcal L}_{\alpha,t}(x+h,y)-K^{\mathcal L}_{\alpha,t}(x,y)\Big|\leq C_{N}\Big(\frac{|h|}{\sqrt{t^{1/\alpha}}}\Big)^{\delta}\frac{t}{(\sqrt{t^{1/\alpha}}+|x-y|)^{n+2\alpha}}
\Big(1+\frac{\sqrt{t^{1/\alpha}}}{\rho(x}+
    \frac{\sqrt{t^{1/\alpha}}}{\rho(y)}\Big)^{-N}.$$
\end{proposition}

For the kernels $\widetilde{D}^{\mathcal L}_{\alpha,t}(\cdot,\cdot)$ and $D^{\mathcal{L},m}_{\alpha,t}(\cdot,\cdot), m\in\mathbb{Z}_{+}, t>0$, defined by (\ref{eq-1.3}),
the following regularity estimates were obtained by  Li et al. in \cite{Li2}.
\begin{proposition}\label{pro3.12}{\rm (\cite[Proposition 3.3]{Li2})} Let $0<\alpha<1$.
\item{\rm (i)} For every $N$, there is a constant $C_{N}>0$ such that
$$\Big|D_{\alpha,t}^{\mathcal{L},m}(x,y)\Big|\leq \frac{C_{N}t}{(\sqrt{t^{1/\alpha}}+|x-y|)^{n+2\alpha }}\Big(1+\frac{\sqrt{t^{1/\alpha}}}{\rho(x)}+\frac{\sqrt{t^{1/\alpha}}}{\rho(y)}\Big)^{-N}.$$
\item{\rm (ii)} Let $0<\delta<\min\{2\alpha,\delta_{0},1\}$. For every $N>0$, there exists a constant $C_{N}>0$ such that, for all $|h|<\sqrt{t^{1/\alpha}}$,
$$\Big|D^{\mathcal{L},m}_{\alpha,t}(x+h,y)-D^{\mathcal{L},m}_{\alpha,t}(x,y)\Big|\leq C_{N}\Big(\frac{h}{\sqrt{t^{1/\alpha}}}\Big)^{\delta}\frac{t}{(\sqrt{t^{1/\alpha}}+|x-y|)^{n+2\alpha }}
\Big(1+\frac{\sqrt{t^{1/\alpha}}}{\rho(x)}+\frac{\sqrt{t^{1/\alpha}}}{\rho(y)}\Big)^{-N}.$$
\item{\rm (iii)} There exists a constant $C_{N}>0$ such that
$$\Big|\int_{\mathbb{R}^{n}}D^{\mathcal{L},m}_{\alpha,t}(x,y)dy\Big|\leq \frac{C_{N}(\sqrt{t^{1/\alpha}}/\rho(x))^{\delta}}{(1+\sqrt{t^{1/\alpha}}/\rho(x))^{N}}.$$
\end{proposition}


\begin{proposition}\label{pro3.13}{\rm (\cite[Propositions 3.6, \ 3.9\ and\ 3.10]{Li2})}     Suppose that  and $V\in B_{q}$ for some $q>n$.
\item{\rm (i)} Let $\alpha\in(0,1)$. For every $N$, there is a constant $C_{N}>0$ such that
$$\Big|\widetilde{D}_{\alpha,t}^{\mathcal{L}}(x,y)\Big|\leq \frac{C_{N}t}{(\sqrt{t^{1/\alpha}}+|x-y|)^{n+2\alpha }}\Big(1+\frac{\sqrt{t^{1/\alpha}}}{\rho(x)}+\frac{\sqrt{t^{1/\alpha}}}{\rho(y)}\Big)^{-N}.$$
\item{\rm (ii)} Let $\alpha\in(0,1)$ and $\delta_{1}=1-n/q$. For every $N>0$, there exists a constant $C_{N}>0$ such that for all $|h|<\sqrt{t^{1/\alpha}}$,
$$\Big|\widetilde{D}^{\mathcal{L}}_{\alpha,t}(x+h,y)-\widetilde{D}^{\mathcal{L}}_{\alpha,t}(x,y)\Big|\leq C_{N}\Big(\frac{h}{\sqrt{t^{1/\alpha}}}\Big)^{\delta_{1}}\frac{t}{(\sqrt{t^{1/\alpha}}+|x-y|)^{n+2\alpha }}
\Big(1+\frac{\sqrt{t^{1/\alpha}}}{\rho(x)}+\frac{\sqrt{t^{1/\alpha}}}{\rho(y)}\Big)^{-N}.$$
\item{\rm (iii)} Let $\alpha\in (0, 1/2-n/2q)$. There exists a constant $C_{N}>0$ such that
$$\Big|\int_{\mathbb{R}^{n}}\widetilde{D}^{\mathcal{L}}_{\alpha,t}(x,y)dy\Big|\leq C_{N}\min\Bigg\{\Big(\frac{\sqrt{t^{1/\alpha}}}{\rho(x)}\Big)^{1+2\alpha},\
 \Big(\frac{\sqrt{t^{1/\alpha}}}{\rho(x)}\Big)^{-N}\Bigg\}.$$
\end{proposition}

In order to establish the $BMO^{\gamma}_{L}$-boundedness of operators via $T1$ type theorem, we need the following propositions.
\begin{proposition}\label{pro1}
There exists a constant $C>0$ such that
\begin{equation*}
  \Big|K^{\mathcal{L}}_{\alpha,t}(x,y)-K_{\alpha,t}(x-y)\Big|\leq
  \left\{\begin{aligned}
  &C\Big(\frac{|x-y|}{\rho(x)}\Big)^{\delta_{0}}\frac{t}{(|x-y|+\sqrt{t^{1/\alpha}})^{n+2\alpha}},\ \ \ \sqrt{t^{1/\alpha}}\leq |x-y|;\\
  &C\Big(\frac{\sqrt{t^{1/\alpha}}}{\rho(x)}\Big)^{\delta_{0}}\frac{t}{(|x-y|+\sqrt{t^{1/\alpha}})^{n+2\alpha}},\ \ \ \sqrt{t^{1/\alpha}}\geq |x-y|.
  \end{aligned}\right.
\end{equation*}
\end{proposition}
\begin{proof}
By the subordinative formula (\ref{p-1}) and Lemma {\ref{lem1}}, we obtain
\begin{eqnarray*}
 \Big|K^{\mathcal{L}}_{\alpha,t}(x,y)-K_{\alpha,t}(x-y)\Big|  &\leq& \int^{\infty}_{0}\eta^{\alpha}_{t}(s)\Big|K^{\mathcal{L}}_{s}(x,y)-K_{s}(x-y)\Big|ds  \\
   &\leq& C\int^{\infty}_{0}\eta^{\alpha}_{t}(s)\Big(\frac{\sqrt{s}}{\rho(x)}\Big)^{\delta_{0}}s^{-n/2}e^{-|x-y|^{2}/s}ds.
\end{eqnarray*}
On the one hand, letting $s=t^{1/\alpha}u$, we can get
\begin{eqnarray*}
  \Big|K^{\mathcal{L}}_{\alpha,t}(x,y)-K_{\alpha,t}(x-y)\Big| &\leq& C\int^{\infty}_{0}\frac{t}{s^{1+\alpha}}\Big(\frac{\sqrt{s}}{\rho(x)}\Big)^{\delta_{0}}s^{-n/2}e^{-|x-y|^{2}/s}ds \\
   &\leq&C\int^{\infty}_{0}\frac{t}{(t^{1/\alpha}u)^{1+\alpha}}\Big(\frac{\sqrt{t^{1/\alpha}u}}{\rho(x)}\Big)^{\delta_{0}}(t^{1/\alpha}u)^{-n/2}
   e^{-c|x-y|^{2}/(t^{1/\alpha}u)}t^{1/\alpha}du  \\
   &\leq& C\Big(\frac{\sqrt{t^{1/\alpha}}}{\rho(x)}\Big)^{\delta_{0}}t^{-{n}/{(2\alpha)}}\int^{\infty}_{0}u^{-1-\alpha-n/2+\delta_{0}/2}e^{-|x-y|^{2}/(t^{1/\alpha}u)}
   du.
\end{eqnarray*}
Applying the change of variables: ${|x-y|^{2}}/{(t^{1/\alpha}u)}=r^{2}$, we deduce that
\begin{eqnarray*}
  \Big|K^{\mathcal{L}}_{\alpha,t}(x,y)-K_{\alpha,t}(x-y)\Big| &\leq& C\Big(\frac{\sqrt{t^{1/\alpha}}}{\rho(x)}\Big)^{\delta_{0}}t^{-n/(2\alpha)}\int^{\infty}_{0}
  \Big(\frac{|x-y|^{2}}{t^{1/\alpha}r^{2}}\Big)^{-1-\alpha+\delta_{0}/2-n/2}e^{-r^{2}}\frac{|x-y|^{2}}{t^{1/\alpha}r^{3}}dr  \\
   &\leq& C\Big(\frac{\sqrt{t^{1/\alpha}}}{\rho(x)}\Big)^{\delta_{0}}|x-y|^{-2\alpha+\delta_{0}-n}t^{1-\delta_{0}/(2\alpha)}\int^{\infty}_{0}
   e^{-r^{2}}r^{2\alpha
   -\delta_{0}+n+1}dr \\
   &\leq& C\Big(\frac{\sqrt{t^{1/\alpha}}}{\rho(x)}\Big)^{\delta_{0}}\frac{t^{1-\delta_{0}/(2\alpha)}}{|x-y|^{2\alpha+n-\delta_{0}}}.
\end{eqnarray*}
On the other hand, taking $\tau=s/t^{1/\alpha}$, we obtain
\begin{eqnarray*}
  \Big|K^{\mathcal{L}}_{\alpha,t}(x,y)-K_{\alpha,t}(x-y)\Big|  &\leq& C\int^{\infty}_{0}\Big(\frac{\sqrt{s}}{\rho(x)}\Big)^{\delta_{0}}s^{-n/2}t^{-1/\alpha}\eta^{\alpha}_{1}(s/t^{1/\alpha})ds  \\
   &\leq&C\int^{\infty}_{0}\Big(\frac{\sqrt{t^{1/\alpha}\tau}}{\rho(x)}\Big)^{\delta_{0}}(t^{1/\alpha}\tau)^{-n/2}\eta^{\alpha}_{1}(\tau)d\tau  \\
   &\leq&C\Big(\frac{\sqrt{t^{1/\alpha}}}{\rho(x)}\Big)^{\delta_{0}}t^{-n/2\alpha}\int^{\infty}_{0}\eta^{\alpha}_{1}(\tau)\tau^{\delta_{0}/2-n/2}d\tau  \\
   &\leq& C \Big(\frac{\sqrt{t^{1/\alpha}}}{\rho(x)}\Big)^{\delta_{0}}t^{-n/2\alpha}.
\end{eqnarray*}
If $\sqrt{t^{1/\alpha}}\leq |x-y|$, then
\begin{eqnarray*}
  \Big|K^{\mathcal{L}}_{\alpha,t}(x,y)-K_{\alpha,t}(x-y)\Big| &\leq& C\Big(\frac{|x-y|}{\rho(x)}\Big)^{\delta_{0}}\frac{t}{|x-y|^{2\alpha+n}}  \\
   &\leq& C\Big(\frac{|x-y|}{\rho(x)}\Big)^{\delta_{0}}\frac{t}{(|x-y|+\sqrt{t^{1/\alpha}})^{2\alpha+n}}.
\end{eqnarray*}
If $\sqrt{t^{1/\alpha}}>|x-y|$, we can see that
\begin{eqnarray*}
  \Big|K^{\mathcal{L}}_{\alpha,t}(x,y)-K_{\alpha,t}(x-y)\Big| &\leq& C\Big(\frac{\sqrt{t^{1/\alpha}}}{\rho(x)}\Big)^{\delta_{0}}\frac{t}{t^{n/(2\alpha)+1}}  \\
   &\leq& C\Big(\frac{\sqrt{t^{1/\alpha}}}{\rho(x)}\Big)^{\delta_{0}}\frac{t}{(|x-y|+\sqrt{t^{1/\alpha}})^{2\alpha+n}}.
\end{eqnarray*}
\end{proof}

Let $\widetilde{D}_{\alpha,t}(\cdot)=t^{1/(2\alpha)}\nabla_{x}e^{-t(-\Delta)^{\alpha}}(\cdot).$ Similar to the proof of Proposition \ref{pro1}, by (\ref{p-1}) and Lemma \ref{lem1-1}, we have

\begin{proposition}\label{pro1-1}
There exists a constant $C>0$ such that
\begin{equation*}
  \Big|\widetilde{D}^{\mathcal{L}}_{\alpha,t}(x,y)-\widetilde{D}_{\alpha,t}(x-y)\Big|\leq
  \left\{\begin{aligned}
  &C\Big(\frac{|x-y|}{\rho(x)}\Big)^{\delta_{0}}\frac{t}{(|x-y|+\sqrt{t^{1/\alpha}})^{n+2\alpha}},\ \ \ \sqrt{t^{1/\alpha}}\leq |x-y|;\\
  &C\Big(\frac{\sqrt{t^{1/\alpha}}}{\rho(x)}\Big)^{\delta_{0}}\frac{t}{(|x-y|+\sqrt{t^{1/\alpha}})^{n+2\alpha}},\ \ \ \sqrt{t^{1/\alpha}}\geq |x-y|.
  \end{aligned}\right.
\end{equation*}

\end{proposition}
\begin{proof}
By the subordinative formula (\ref{p-1}) and Lemma {\ref{lem1-1}}, we obtain
\begin{eqnarray*}
 \Big|\widetilde{D}^{\mathcal{L}}_{\alpha,t}(x,y)-\widetilde{D}_{\alpha,t}(x-y)\Big|  &\leq& \sqrt{t^{1/\alpha}}\int^{\infty}_{0}\eta^{\alpha}_{t}(s)\Big|\nabla_{x}K^{\mathcal{L}}_{s}(x,y)-\nabla_{x}K_{s}(x-y)\Big|ds  \\
   &\leq& C\sqrt{t^{1/\alpha}}\int^{\infty}_{0}\eta^{\alpha}_{t}(s)\Big(\frac{\sqrt{s}}{\rho(x)}\Big)^{\delta_{0}}s^{-(n+1)/2}e^{-c|x-y|^{2}/s}ds.
\end{eqnarray*}
On the one hand, letting $s=t^{1/\alpha}u$, we have
\begin{eqnarray*}
  \Big|\widetilde{D}^{\mathcal{L}}_{\alpha,t}(x,y)-\widetilde{D}_{\alpha,t}(x-y)\Big| &\leq& C\sqrt{t^{1/\alpha}}\int^{\infty}_{0}\frac{t}{s^{1+\alpha}}\Big(\frac{\sqrt{s}}{\rho(x)}\Big)^{\delta_{0}}s^{-(n+1)/2}e^{-c|x-y|^{2}/s}ds \\
   &\leq&C\sqrt{t^{1/\alpha}}\int^{\infty}_{0}\frac{t}{(t^{1/\alpha}u)^{1+\alpha}}\Big(\frac{\sqrt{t^{1/\alpha}u}}{\rho(x)}\Big)^{\delta_{0}}
   (t^{1/\alpha}u)^{-(n+1)/2}
   e^{-c|x-y|^{2}/(t^{1/\alpha}u)}t^{1/\alpha}du  \\
   &\leq& C\sqrt{t^{1/\alpha}}\Big(\frac{\sqrt{t^{1/\alpha}}}{\rho(x)}\Big)^{\delta_{0}}t^{-(n+1)/(2\alpha)}\int^{\infty}_{0}u^{-1-\alpha-(n+1)/2+\delta_{0}/2}
   e^{-c|x-y|^{2}/(t^{1/\alpha}u)}
   du.
\end{eqnarray*}
By the change of variables ${|x-y|^{2}}/{(t^{1/\alpha}u)}=r^{2}$, we can get
\begin{eqnarray*}
  \Big|\widetilde{D}^{\mathcal{L}}_{\alpha,t}(x,y)-\widetilde{D}_{\alpha,t}(x-y)\Big| &\leq& C\Big(\frac{\sqrt{t^{1/\alpha}}}{\rho(x)}\Big)^{\delta_{0}}t^{-n/2\alpha}\int^{\infty}_{0}
  \Big(\frac{|x-y|^{2}}{t^{1/\alpha}r^{2}}\Big)^{-1-\alpha+\delta_{0}/2-(n+1)/2}e^{-r^{2}}\frac{|x-y|^{2}}{t^{1/\alpha}r^{3}}dr  \\
   &\leq& C\Big(\frac{\sqrt{t^{1/\alpha}}}{\rho(x)}\Big)^{\delta_{0}}|x-y|^{-2\alpha+\delta_{0}-n-1}t^{1-\delta_{0}/(2\alpha)+1/(2\alpha)}
   \int^{\infty}_{0}e^{-r^{2}}r^{2\alpha
   -\delta_{0}+n}dr \\
   &\leq& C\Big(\frac{\sqrt{t^{1/\alpha}}}{\rho(x)}\Big)^{\delta_{0}}\frac{t^{1+1/(2\alpha)-\delta_{0}/(2\alpha)}}{|x-y|^{2\alpha+n+1-\delta_{0}}}.
\end{eqnarray*}
On the other hand, the change of variable $\tau=s/t^{1/\alpha}$ gives
\begin{eqnarray*}
  \Big|\widetilde{D}^{\mathcal{L}}_{\alpha,t}(x,y)-\widetilde{D}_{\alpha,t}(x-y)\Big|  &\leq& C\sqrt{t^{1/\alpha}}\int^{\infty}_{0}\Big(\frac{\sqrt{s}}{\rho(x)}\Big)^{\delta_{0}}s^{-(n+1)/2}t^{-1/\alpha}\eta^{\alpha}_{1}(s/t^{1/\alpha})ds  \\
   &\leq&C\sqrt{t^{1/\alpha}}\int^{\infty}_{0}\Big(\frac{\sqrt{t^{1/\alpha}\tau}}{\rho(x)}\Big)^{\delta_{0}}(t^{1/\alpha}\tau)^{-(n+1)/2}
   \eta^{\alpha}_{1}(\tau)d\tau  \\
   &\leq&C\Big(\frac{\sqrt{t^{1/\alpha}}}{\rho(x)}\Big)^{\delta_{0}}t^{-n/(2\alpha)}\int^{\infty}_{0}\eta^{\alpha}_{1}(\tau)\tau^{\delta_{0}/2-(n+1)/2}d\tau  \\
   &\leq& C \Big(\frac{\sqrt{t^{1/\alpha}}}{\rho(x)}\Big)^{\delta_{0}}t^{-n/(2\alpha)}.
\end{eqnarray*}
If $\sqrt{t^{1/\alpha}}\leq |x-y|$, then
\begin{eqnarray*}
  \Big|\widetilde{D}^{\mathcal{L}}_{\alpha,t}(x,y)-\widetilde{D}_{\alpha,t}(x-y)\Big| &\leq& C\Big(\frac{|x-y|}{\rho(x)}\Big)^{\delta_{0}}\frac{t}{|x-y|^{2\alpha+n}}  \\
   &\leq& C\Big(\frac{|x-y|}{\rho(x)}\Big)^{\delta_{0}}\frac{t}{(|x-y|+\sqrt{t^{1/\alpha}})^{2\alpha+n}}.
\end{eqnarray*}
If $t^{1/2\alpha}>|x-y|$, we can also get
\begin{eqnarray*}
  \Big|\widetilde{D}^{\mathcal{L}}_{\alpha,t}(x,y)-\widetilde{D}_{\alpha,t}(x-y)\Big| &\leq& C\Big(\frac{\sqrt{t^{1/\alpha}}}{\rho(x)}\Big)^{\delta_{0}}\frac{t}{t^{n/(2\alpha)+1}}  \\
   &\leq& C\Big(\frac{\sqrt{t^{1/\alpha}}}{\rho(x)}\Big)^{\delta_{0}}\frac{t}{(|x-y|+\sqrt{t^{1/\alpha}})^{2\alpha+n}}.
\end{eqnarray*}

\end{proof}
Let $D^{m}_{\alpha,t}(\cdot)=t^{m}\partial^{m}_{t}e^{-t(-\Delta)^{\alpha}}(\cdot).$ We have
\begin{proposition}\label{pro1-1-1}
There exists a constant $C>0$ such that
\begin{equation*}
  \Big|D^{\mathcal{L},m}_{\alpha,t}(x,y)-D^{m}_{\alpha,t}(x-y)\Big|\leq
 \left\{ \begin{aligned}
  &C\Big(\frac{|x-y|}{\rho(x)}\Big)^{\delta_{0}}\frac{t}{(|x-y|+\sqrt{t^{1/\alpha}})^{n+2\alpha}},\ \ \ \sqrt{t^{1/\alpha}}\leq |x-y|;\\
  &C\Big(\frac{\sqrt{t^{1/\alpha}}}{\rho(x)}\Big)^{\delta_{0}}\frac{t}{(|x-y|+\sqrt{t^{1/\alpha}})^{n+2\alpha}},\ \ \ \sqrt{t^{1/\alpha}}\geq |x-y|.
  \end{aligned}\right.
\end{equation*}

\end{proposition}
\begin{proof}
The proposition can be proved by Proposition \ref{pro2.5} and (\ref{p-1}). Since the argument is similar to that of Proposition \ref{pro1}, the details is omitted.

\end{proof}

\begin{proposition}\label{pro2}
For every $0<\delta<\min\{2\alpha,\delta_{0}\}$ and $|y-z|<\min\{C\rho(y),|x-y|/4\}$, there exists a constant $C>0$ such that
$$\Big|\Big(K^{\mathcal{L}}_{\alpha,t}(x,y)-K_{\alpha,t}(x-y)\Big)-\Big(K^{\mathcal{L}}_{\alpha,t}(x,z)-K_{\alpha,t}(x-z)\Big)\Big|\leq C\Big(\frac{|y-z|}{\rho(x)}\Big)^{\delta}
\frac{t}{(\sqrt{t^{1/\alpha}}+|x-y|)^{n+2\alpha}}.$$
\end{proposition}
\begin{proof}
By the subordinative formula (\ref{p-1}), we can use Lemma \ref{lem2} to deduce
\begin{eqnarray*}
  &&\Big|\Big(K^{\mathcal{L}}_{\alpha,t}(x,y)-K_{\alpha,t}(x-y)\Big)-\Big(K^{\mathcal{L}}_{\alpha,t}(x,z)-K_{\alpha,t}(x-z)\Big)\Big|\\
  &&\leq \int^{\infty}_{0}\eta^{\alpha}_{t}(s)\Big|\Big(K^{\mathcal{L}}_{s}(x,y)-K_{s}(x-y)\Big)-\Big(K^{\mathcal{L}}_{s}(x,z)-K_{s}(x-z)\Big)\Big|ds  \\
   &&\leq  C\int^{\infty}_{0}\eta^{\alpha}_{t}(s)\Big(\frac{|y-z|}{\rho(x)}\Big)^{\delta}s^{-n/2}e^{-|x-y|^{2}/s}ds.
\end{eqnarray*}
On the one hand, taking $s=t^{1/\alpha}u$, we obtain
\begin{eqnarray*}
&&\Big|\Big(K^{\mathcal{L}}_{\alpha,t}(x,y)-K_{\alpha,t}(x-y)\Big)-\Big(K^{\mathcal{L}}_{\alpha,t}(x,z)-K_{\alpha,t}(x-z)\Big)\Big|\\ &&\leq
\int^{\infty}_{0}\frac{Ct}{s^{1+\alpha}}\Big(\frac{|y-z|}{\rho(x)}\Big)^{\delta}s^{-n/2}e^{-|x-y|^{2}/s}ds  \\
   &&\leq  C\int^{\infty}_{0}\frac{t}{(t^{1/\alpha}u)^{1+\alpha}}\Big(\frac{|y-z|}{\rho(x)}\Big)^{\delta}(t^{1/\alpha}u)^{-n/2}
   e^{-|x-y|^{2}/(t^{1/\alpha}u)}t^{1/\alpha}du \\
   &&\leq  C\Big(\frac{|y-z|}{\rho(x)}\Big)^{\delta}t^{-n/(2\alpha)}\int^{\infty}_{0}e^{-|x-y|^{2}/(t^{1/\alpha}u)}u^{-1-\alpha-n/2}du.
\end{eqnarray*}
Let $\displaystyle \frac{|x-y|^{2}}{(t^{1/\alpha}u)}=r$. We can see that
\begin{eqnarray*}
 &&\Big|\Big(K^{\mathcal{L}}_{\alpha,t}(x,y)-K_{\alpha,t}(x-y)\Big)-\Big(K^{\mathcal{L}}_{\alpha,t}(x,z)-K_{\alpha,t}(x-z)\Big)\Big|\\  &&\leq  C\Big(\frac{|y-z|}{\rho(x)}\Big)^{\delta}t^{-n/(2\alpha)}\int^{\infty}_{0}e^{-r}\Big(\frac{|x-y|^{2}}{t^{1/\alpha}r}\Big)^{-1-\alpha-n/2}
 \frac{|x-y|^{2}}{t^{1/\alpha}r^{2}}dr \\
   &&\leq  C\Big(\frac{|y-z|}{\rho(x)}\Big)^{\delta}\frac{t}{|x-y|^{2\alpha+n}}.
\end{eqnarray*}
On the other hand, letting $\displaystyle \frac{s}{t^{1/\alpha}}=\tau$, we obtain
\begin{eqnarray*}
 &&\Big|\Big(K^{\mathcal{L}}_{\alpha,t}(x,y)-K_{\alpha,t}(x-y)\Big)-\Big(K^{\mathcal{L}}_{\alpha,t}(x,z)-K_{\alpha,t}(x-z)\Big)\Big| \\ &&\leq  \int^{\infty}_{0}\frac{1}{t^{1/\alpha}}\eta^{\alpha}_{1}(s/t^{1/\alpha})\Big(\frac{|y-z|}{\rho(x)}\Big)^{\delta}s^{-n/2}ds  \\
   &&\leq  C\Big(\frac{|y-z|}{\rho(x)}\Big)^{\delta}\int^{\infty}_{0}\eta^{\alpha}_{1}(\tau)(t^{1/\alpha}\tau)^{-n/2}d\tau \\
   &&\leq  C\Big(\frac{|y-z|}{\rho(x)}\Big)^{\delta}t^{-n/(2\alpha)}.
\end{eqnarray*}
{\it Case 1: $\sqrt{t^{1/\alpha}}\leq|x-y|$.} We obtain
\begin{eqnarray*}
 &&\Big|\Big(K^{\mathcal{L}}_{\alpha,t}(x,y)-K_{\alpha,t}(x-y)\Big)-\Big(K^{\mathcal{L}}_{\alpha,t}(x,z)-K_{\alpha,t}(x-z)\Big)\Big|\\
   &&\leq  C\Big(\frac{|y-z|}{\rho(x)}\Big)^{\delta}
\frac{t}{(\sqrt{t^{1/\alpha}}+|x-y|)^{n+2\alpha}}.
\end{eqnarray*}
{\it Case 2: $\sqrt{t^{1/\alpha}}>|x-y|$.} We can see that
\begin{eqnarray*}
  &&\Big|\Big(K^{\mathcal{L}}_{\alpha,t}(x,y)-K_{\alpha,t}(x-y)\Big)-\Big(K^{\mathcal{L}}_{\alpha,t}(x,z)-K_{\alpha,t}(x-z)\Big)\Big|\\
   &&\leq C\Big(\frac{|y-z|}{\rho(x)}\Big)^{\delta}
\frac{t}{t^{n/(2\alpha)+1}}\leq C\Big(\frac{|y-z|}{\rho(x)}\Big)^{\delta}\frac{t}{(\sqrt{t^{1/\alpha}}+|x-y|)^{n+2\alpha}}.
\end{eqnarray*}
\end{proof}

\begin{proposition}\label{pro2-2-2}
For every $0<\delta<\min\{2\alpha,\delta_{0}\}$ and $|y-z|<\min\{C\rho(y),|x-y|/4\}$, there exists a constant $C>0$ such that
$$\Big|\Big(D^{\mathcal{L},m}_{\alpha,t}(x,y)-D^{m}_{\alpha,t}(x-y)\Big)-\Big(D^{\mathcal{L},m}_{\alpha,t}(x,z)-D^{m}_{\alpha,t}(x-z)\Big)\Big|\leq C\Big(\frac{|y-z|}{\rho(x)}\Big)^{\delta}
\frac{t}{(\sqrt{t^{1/\alpha}}+|x-y|)^{n+2\alpha}}.$$
\end{proposition}
\begin{proof}
The proof is similar to that of Proposition \ref{pro2}, so we omit the details.
\end{proof}

\begin{proposition}\label{pro2-2}
Suppose that $V\in B_{q}$ for some $q>n$. For $0<\delta'<\delta_{1}:=1-n/q$, and $|y-z|<\min\{C\rho(y),|x-y|/4\}$, there exists a constant $C>0$ such that
$$\Big|\Big(\widetilde{D}^{\mathcal{L}}_{\alpha,t}(x,y)-\widetilde{D}_{\alpha,t}(x-y)\Big)-
\Big(\widetilde{D}^{\mathcal{L}}_{\alpha,t}(x,z)-\widetilde{D}_{\alpha,t}(x-z)\Big)\Big|\leq C\Big(\frac{|y-z|}{\rho(x)}\Big)^{\delta'}
\frac{t}{(\sqrt{t^{1/\alpha}}+|x-y|)^{n+2\alpha}}.$$
\end{proposition}
\begin{proof}
By the subordinative formula (\ref{p-1}), we can use Lemma \ref{lem2-2} to get
\begin{eqnarray*}
  &&\Big|\Big(\widetilde{D}^{\mathcal{L}}_{\alpha,t}(x,y)-\widetilde{D}_{\alpha,t}(x-y)\Big)-
  \Big(\widetilde{D}^{\mathcal{L}}_{\alpha,t}(x,z)-\widetilde{D}_{\alpha,t}(x-z)\Big)\Big|\\
  &&\leq t^{1/(2\alpha)}\int^{\infty}_{0}\eta^{\alpha}_{t}(s)
  \Big|\Big(\nabla_{x}K^{\mathcal{L}}_{s}(x,y)-\nabla_{x}K_{s}(x-y)\Big)-\nabla_{x}\Big(K^{\mathcal{L}}_{s}(x,z)-K_{s}(x-z)\Big)\Big|ds  \\
   &&\leq  Ct^{1/(2\alpha)}\int^{\infty}_{0}\eta^{\alpha}_{t}(s)\Big(\frac{|y-z|}{\rho(x)}\Big)^{\delta'}s^{-(n+1)/2}e^{-c|x-y|^{2}/s}ds.
\end{eqnarray*}
On the one hand, taking $s=t^{1/\alpha}u$, we obtain
\begin{eqnarray*}
&&\Big|\Big(\widetilde{D}^{\mathcal{L}}_{\alpha,t}(x,y)-\widetilde{D}_{\alpha,t}(x-y)\Big)-
\Big(\widetilde{D}^{\mathcal{L}}_{\alpha,t}(x,z)-\widetilde{D}_{\alpha,t}(x-z)\Big)\Big|\\ &&\leq
Ct^{1/(2\alpha)}\int^{\infty}_{0}\frac{t}{s^{1+\alpha}}\Big(\frac{|y-z|}{\rho(x)}\Big)^{\delta'}s^{-(n+1)/2}e^{-c|x-y|^{2}/s}ds  \\
   &&\leq  Ct^{1/(2\alpha)}\int^{\infty}_{0}\frac{t}{(t^{1/\alpha}u)^{1+\alpha}}\Big(\frac{|y-z|}{\rho(x)}\Big)^{\delta'}(t^{1/\alpha}u)^{-(n+1)/2}
   e^{-c|x-y|^{2}/(t^{1/\alpha}u)}t^{1/\alpha}du \\
   &&\leq  C\Big(\frac{|y-z|}{\rho(x)}\Big)^{\delta'}t^{-n/(2\alpha)}\int^{\infty}_{0}e^{-c|x-y|^{2}/(t^{1/\alpha}u)}u^{-1-\alpha-(n+1)/2}du.
\end{eqnarray*}
Let ${|x-y|^{2}}/{(t^{1/\alpha}u)}=r$. It holds
\begin{eqnarray}\label{eq-3.1}
 &&\Big|\Big(\widetilde{D}^{\mathcal{L}}_{\alpha,t}(x,y)-\widetilde{D}_{\alpha,t}(x-y)\Big)-
 \Big(\widetilde{D}^{\mathcal{L}}_{\alpha,t}(x,z)-\widetilde{D}_{\alpha,t}(x-z)\Big)\Big|\\
   &&\leq  C\Big(\frac{|y-z|}{\rho(x)}\Big)^{\delta'}t^{-n/(2\alpha)}\int^{\infty}_{0}e^{-r}\Big(\frac{|x-y|^{2}}{t^{1/\alpha}r}\Big)^{-1-\alpha-(n+1)/2}
 \frac{|x-y|^{2}}{t^{1/\alpha}r^{2}}dr \nonumber\\
   &&\leq  C\Big(\frac{|y-z|}{\rho(x)}\Big)^{\delta'}\frac{t^{1+1/(2\alpha)}}{|x-y|^{2\alpha+n+1}}.\nonumber
\end{eqnarray}
On the other hand, letting $s/t^{1/\alpha}=\tau$, we obtain
\begin{eqnarray}\label{eq-3.2}
 &&\Big|\Big(\widetilde{D}^{\mathcal{L}}_{\alpha,t}(x,y)-\widetilde{D}_{\alpha,t}(x-y)\Big)-
 \Big(\widetilde{D}^{\mathcal{L}}_{\alpha,t}(x,z)-\widetilde{D}_{\alpha,t}(x-z)\Big)\Big| \\
  &&\leq  Ct^{1/(2\alpha)}\int^{\infty}_{0}\frac{1}{t^{1/\alpha}}\eta^{\alpha}_{1}(s/t^{1/\alpha})\Big(\frac{|y-z|}{\rho(x)}\Big)^{\delta'}s^{-(n+1)/2}ds\nonumber  \\
   &&\leq  Ct^{1/(2\alpha)}\Big(\frac{|y-z|}{\rho(x)}\Big)^{\delta'}\int^{\infty}_{0}\eta^{\alpha}_{1}(\tau)(t^{1/\alpha}\tau)^{-(n+1)/2}d\tau\nonumber \\
   &&\leq  C\Big(\frac{|y-z|}{\rho(x)}\Big)^{\delta'}t^{-n/(2\alpha)}.\nonumber
\end{eqnarray}
It is easy to see that the desired estimate follows from (\ref{eq-3.1})\ and \ (\ref{eq-3.2}). The proof of Proposition \ref{pro2-2} is finished.

\end{proof}

\section{$BMO^{\gamma}_{\mathcal{L}}$-boundedness via $T1$ theorem}\label{sec-4}
\subsection{Maximal operators for the fractional heat semigroup}\label{sec-4-1}
\begin{definition}\label{defc-1}
Let $0<\gamma\leq 1$. The Campanato type space $BMO^{\gamma}_{\mathcal{L},L^{\infty}((0,\infty),dt)}(\mathbb{R}^{n})$ is defined as the set of all  locally integrable functions $f$ satisfying
$$\|f\|_{BMO^{\gamma}_{\mathcal{L},L^{\infty}((0,\infty),dt)}}:=\sup_{B\subset\mathbb{R}^{n}}
\Bigg\{\frac{1}{|B|^{1+\gamma/n}}\int_{B}\|f(x,t)-f(B,V)\|_{L^{\infty}((0,\infty),dt)}dx\Bigg\}<\infty.$$
\end{definition}

 To prove that $M^{\alpha}_{\mathcal{L}}$ is bounded from $BMO_{\mathcal{L}}^{\gamma}(\mathbb R^{n})$ into itself, we give a vector-valued interpretation of the operator and apply Lemma \ref{lem4.1}. Indeed, it is clear that $M^{\alpha}_{\mathcal{L}}f=\|e^{-t\mathcal{L}^{\alpha}}f\|_{L^{\infty}((0,\infty),dt)}$. Hence, it is enough to show that the operator $\Lambda(f):=\{e^{-t\mathcal{L}^{\alpha}}f\}_{t>0}$ is bounded from $BMO^{\gamma}_{\mathcal{L}}$ into $BMO_{\mathcal{L},L^{\infty}((0,\infty),dt)}^{\gamma}$.

By the spectral theorem, $\Lambda$ is bounded from $L^{2}(\mathbb{R}^{n})$ into $L^{2}_{L^{\infty}((0,\infty),dt)}(\mathbb{R}^{n})$. The desired result can be then deduced from the following theorem.

\begin{theorem}\label{the3}
Assume that the potential $V\in B_{q}$ with $q>n/2$. Let $x,y,z\in\mathbb{R}^{n}$.
\item{\rm (i)} For any $N>0$, there exists a constant $C_{N}$ such that
 $$\Big\|K^{\mathcal{L}}_{\alpha,t}(x,y)\Big\|_{L^{\infty}((0,\infty),dt)}\leq \frac{C_{N}}{|x-y|^{n}}\Big(1+\frac{|x-y|}{\rho(x)}+\frac{|x-y|}{\rho(y)}\Big)^{-N}.$$
\item{\rm (ii)} For $|x-y|>2|y-z|$ and any $0<\delta<\min\{2\alpha,\delta_{0}\}$, there exists a constant $C>0$ such that
    \begin{multline}\label{t-1}
       \Big\|K^{\mathcal{L}}_{\alpha,t}(x,y)-K^{\mathcal{L}}_{\alpha,t}(x,z)\Big\|_{L^{\infty}((0,\infty),dt)}+
    \Big\|K^{\mathcal{L}}_{\alpha,t}(y,x)-K^{\mathcal{L}}_{\alpha,t}(z,x)\Big\|_{L^{\infty}((0,\infty),dt)}\\
     \quad\leq \frac{C|y-z|^{\delta}}{|x-y|^{n+\delta}}.
    \end{multline}
\item{\rm (iii)} There exists a constant $C$ such that for every ball $B=B(x,r)$ with $0<r\leq \rho(x)/2$,
$$\log\Big(\frac{\rho(x)}{r}\Big)\frac{1}{|B|}\int_{B}\Big\|e^{-t\mathcal{L}^{\alpha}}1(y)-(e^{-t\mathcal{L}^{\alpha}}1)_{B}\Big\|_{L^{\infty}((0,\infty),dt)}dy\leq C,$$
and, if $\gamma<\min\{2\alpha,1,\delta_{0}\}$ then
$$\Big(\frac{\rho(x)}{r}\Big)^{\gamma}\frac{1}{|B|}\int_{B}\Big\|e^{-t\mathcal{L}^{\alpha}}1(y)-(e^{-t\mathcal{L}^{\alpha}}1)_{B}\Big\|_{L^{\infty}((0,\infty),dt)}dy\leq C.$$
\end{theorem}
\begin{proof}
For (i), from (i) of Proposition \ref{pro2.51}, we can get
$$\Big|K^{\mathcal{L}}_{\alpha,t}(x,y)\Big|\leq C_{N} \min\Bigg\{\frac{t^{1+N/\alpha}}{|x-y|^{n+2\alpha+2N}},\ t^{-n/(2\alpha)}\Bigg\}\Big(1+\frac{\sqrt{t^{1/\alpha}}}{\rho(x)}+\frac{\sqrt{t^{1/\alpha}}}{\rho(y)}\Big)^{-N}.$$
If $\sqrt{t^{1/\alpha}}>|x-y|$, then
\begin{eqnarray*}
  \Big|K^{\mathcal{L}}_{\alpha,t}(x,y)\Big| &\leq& \frac{C_{N}}{(\sqrt{t^{1/\alpha}})^{n}}\Big(1+\frac{|x-y|}{\rho(x)}+\frac{|x-y|}{\rho(y)}\Big)^{-N}  \\
   &\leq&\frac{C}{|x-y|^{n}}\Big(1+\frac{|x-y|}{\rho(x)}+\frac{|x-y|}{\rho(y)}\Big)^{-N}.
\end{eqnarray*}
If $\sqrt{t^{1/\alpha}}\leq |x-y|$, we obtain
\begin{eqnarray*}
  \Big|K^{\mathcal{L}}_{\alpha,t}(x,y)\Big| &\leq& \frac{C_{N}t^{1+N/\alpha}}{|x-y|^{n+2\alpha+2N}}\Big(\frac{\sqrt{t^{1/\alpha}}}{|x-y|}\Big)^{-N}
  \Big(\frac{|x-y|}{\sqrt{t^{1/\alpha}}}+\frac{|x-y|}{\rho(x)}+\frac{|x-y|}{\rho(y)}\Big)^{-N} \\
  &\leq& \frac{C_{N}(\sqrt{t^{1/\alpha}})^{2\alpha+N}}{|x-y|^{n+2\alpha+N}}\Big(1+\frac{|x-y|}{\rho(x)}+\frac{|x-y|}{\rho(y)}\Big)^{-N} \\
  &\leq& \frac{C_{N}}{|x-y|^{n}}\Big(1+\frac{|x-y|}{\rho(x)}+\frac{|x-y|}{\rho(y)}\Big)^{-N}.
\end{eqnarray*}

For (ii), from (ii) of Proposition \ref{pro2.51}, we obtain
$$\Big|K^{\mathcal{L}}_{\alpha,t}(x,y)-K^{\mathcal{L}}_{\alpha,t}(x,z)\Big|\leq C_{N}\min\Bigg\{\frac{t^{1+N/\alpha}|y-z|^{\delta}}{|x-y|^{2\alpha+n+2N+\delta}},\ t^{-n/(2\alpha)}\Bigg\}\Big(\frac{|y-z|}{\sqrt{t^{1/\alpha}}}\Big)^{\delta}.$$
If $\sqrt{t^{1/\alpha}}>|x-y|$, then
\begin{eqnarray*}
 \Big|K^{\mathcal{L}}_{\alpha,t}(x,y)-K^{\mathcal{L}}_{\alpha,t}(x,z)\Big|  &\leq& \frac{C_{N}}{|x-y|^{n}}\Big(\frac{|y-z|}{|x-y|}\Big)^{\delta}
   \leq \frac{C|y-z|^{\delta}}{|x-y|^{n+\delta}}.
\end{eqnarray*}
If $|x-y|>\sqrt{t^{1/\alpha}}$, we can also get
\begin{eqnarray*}
 \Big|K^{\mathcal{L}}_{\alpha,t}(x,y)-K^{\mathcal{L}}_{\alpha,t}(x,z)\Big|  &\leq& \frac{C_{N}|x-y|^{2\alpha+2N}}{|x-y|^{2\alpha+n+2N+\delta}}|y-z|^{\delta}
 \leq \frac{C_{N}|y-z|^{\delta}}{|x-y|^{n+\delta}}.
\end{eqnarray*}
The symmetry of the kernel $K^{L}_{\alpha,t}(\cdot,\cdot)$ gives the conclusion of (ii).

For (iii), letting $B=B(x,r)$ with $0<r\leq \rho(x)/2$, the triangle inequality gives
$$\Big\|e^{-t\mathcal{L}^{\alpha}}1(y)-(e^{-t\mathcal{L}^{\alpha}}1)_{B}\Big\|_{L^{\infty}((0,\infty),dt)}\leq \frac{1}{|B|}\int_{B}\Big\|e^{-t\mathcal{L}^{\alpha}}1(y)-e^{-t\mathcal{L}^{\alpha}}1(z)\Big\|_{L^{\infty}((0,\infty),dt)}dz.$$
We estimate $\|e^{-t\mathcal{L}^{\alpha}}1(y)-e^{-t\mathcal{L}^{\alpha}}1(z)\|_{L^{\infty}((0,\infty),dt)}$. Because $y,z\in B$, $\rho(y)\sim \rho(z)\sim\rho(x).$
By Proposition \ref{pro1}, we split
$|e^{-t\mathcal{L}^{\alpha}}1(y)-e^{-t\mathcal{L}^{\alpha}}1(z)|\leq S_{1}+S_{2}$,
where
\begin{equation*}
\left\{  \begin{aligned}
S_{1}:= \int_{\mathbb{R}^{n}}\Big|K^{\mathcal{L}}_{\alpha,t}(y,w)-K_{\alpha,t}(y,w)\Big|dw;\\
S_{2}:=\int_{\mathbb{R}^{n}}\Big|K^{\mathcal{L}}_{\alpha,t}(z,w)-K_{\alpha,t}(z,w)\Big|dw.
\end{aligned}\right.
\end{equation*}

For $S_{1}$, if $|y-w|\leq\sqrt{t^{1/\alpha}}$, we obtain
\begin{eqnarray*}
  S_{1} \leq\int_{\mathbb{R}^{n}}\Big(\frac{\sqrt{t^{1/\alpha}}}{\rho(x)}\Big)^{\delta_{0}}\frac{t}{(\sqrt{t^{1/\alpha}}+|y-w|)^{n+2\alpha}}dw
   \leq \Big(\frac{\sqrt{t^{1/\alpha}}}{\rho(x)}\Big)^{\delta_{0}}.
\end{eqnarray*}
If $|y-w|>\sqrt{t^{1/\alpha}}$, we can see that
\begin{eqnarray*}
   S_{1}&\leq&\int_{\mathbb{R}^{n}}\Big(\frac{|y-w|}{\rho(x)}\Big)^{\delta_{0}}\frac{t}{(\sqrt{t^{1/\alpha}}+|y-w|)^{n+2\alpha}}dw  \\
   &\leq&\frac{1}{\rho(x)^{\delta_{0}}}\int^{\infty}_{0}\frac{|y-w|^{\delta_{0}+n-1}t}{(|y-w|+\sqrt{t^{1/\alpha}})^{n+2\alpha}}dw  \\
   &\leq&\Big(\frac{\sqrt{t^{1/\alpha}}}{\rho(x)}\Big)^{\delta_{0}}\int^{\infty}_{0}u^{\delta_{0}+n-1}(1+u)^{-n-2\alpha}du  \\
   &\leq& \Big(\frac{\sqrt{t^{1/\alpha}}}{\rho(x)}\Big)^{\delta_{0}},
\end{eqnarray*}
where we have chosen $\delta_{0}<2\alpha$ since $\delta_{0}=2-n/q$, $q>n/2$.
The proof of the term $S_{2}$ is similar to that of the term $S_{1}$, so we omit it. Then we can get
$$\Big|e^{-t\mathcal{L}^{\alpha}}1(y)-e^{-t\mathcal{L}^{\alpha}}1(z)\Big|\leq C\Big(\frac{\sqrt{t^{1/\alpha}}}{\rho(x)}\Big)^{\delta_{0}},$$
which shows that if $\sqrt{t^{1/\alpha}}\leq 2r$,
$$\Big|e^{-t\mathcal{L}^{\alpha}}1(y)-e^{-t\mathcal{L}^{\alpha}}1(z)\Big|\leq C\Big(\frac{r}{\rho(x)}\Big)^{\delta_{0}}.$$
If $\sqrt{t^{1/\alpha}}>2r$, then $|y-z|\leq 2r<\sqrt{t^{1/\alpha}}$. Hence, Proposition \ref{pro2.51} implies that for $0<\delta<\delta_{0}$,
\begin{multline}\label{2}
  \Big|e^{-t\mathcal{L}^{\alpha}}1(y)-e^{-t\mathcal{L}^{\alpha}}1(z)\Big| \leq \int_{\mathbb{R}^{n}}\Big|K^{\mathcal{L}}_{\alpha,t}(y,w)-K^{\mathcal{L}}_{\alpha,t}(z,w)\Big|dw    \\
   \leq \int_{\mathbb{R}^{n}}\Big(\frac{|y-z|}{\sqrt{t^{1/\alpha}}}\Big)^{\delta}\frac{t}{(\sqrt{t^{1/\alpha}}+|y-w|)^{n+2\alpha}}dw \leq C\Big(\frac{|y-z|}{\sqrt{t^{1/\alpha}}}\Big)^{\delta}\leq C \Big(\frac{r}{\sqrt{t^{1/\alpha}}}\Big)^{\delta}.
\end{multline}

Therefore, if $\sqrt{t^{1/\alpha}}>\rho(x)$, (\ref{2}) gives
$$|e^{-t\mathcal{L}^{\alpha}}1(y)-e^{-t\mathcal{L}^{\alpha}}1(z)|\leq C\Big(\frac{r}{\rho(x)}\Big)^{\delta}.$$
When $2r<\sqrt{t^{1/\alpha}}<\rho(x)$, we have
$  |e^{-t\mathcal{L}^{\alpha}}1(y)-e^{-t\mathcal{L}^{\alpha}}1(z)|  =  I+II+III,$
where
\begin{equation*}
 \left\{ \begin{aligned}
 I&:=\int_{|w-y|>c\rho(y)>4|y-z|}\Big|\Big(K^{\mathcal{L}}_{\alpha,t}(y,w)-K_{\alpha,t}(y,w)\Big)-\Big(
   K^{\mathcal{L}}_{\alpha,t}(z,w)-K_{\alpha,t}(z,w)\Big)\Big|dw;\\
 II&:=\int_{4|y-z|<|w-y|<c\rho(y)}\Big|\Big(K^{\mathcal{L}}_{\alpha,t}(y,w)-K_{\alpha,t}(y,w)\Big)-\Big(
   K^{\mathcal{L}}_{\alpha,t}(z,w)-K_{\alpha,t}(z,w)\Big)\Big|dw;\\
 III&:= \int_{|w-y|<4|y-z|}\Big|\Big(K^{\mathcal{L}}_{\alpha,t}(y,w)-K_{\alpha,t}(y,w)\Big)-\Big(
   K^{\mathcal{L}}_{\alpha,t}(z,w)-K_{\alpha,t}(z,w)\Big)\Big|dw.
  \end{aligned}\right.
\end{equation*}

 Noticing that the estimate (\ref{t-1}) is valid for the classical fractional heat kernel. For $I$, by (\ref{t-1}), we can get
$$I\leq C\int_{|w-y|>c\rho(y)>4|y-z|}\frac{|y-z|^{\delta}}{|w-y|^{n+\delta}}dw\leq C\Big(\frac{r}{\rho(x)}\Big)^{\delta}.$$
For $II$, we apply Proposition \ref{pro2} and the fact that $\rho(w)\sim \rho(y)$ in the region of integration.
\begin{eqnarray*}
  II \leq C|y-z|^{\delta}\int_{4|y-z|<|w-y|<c\rho(y)}\frac{tdw}{\rho(w)^{\delta}(\sqrt{t^{1/\alpha}}+|w-y|)^{n+2\alpha}}
   \leq C\Big(\frac{r}{\rho(x)}\Big)^{\delta}.
\end{eqnarray*}

For $III$, since $|y-z|\leq 2r<\sqrt{t^{1/\alpha}},$ we have $|w-y|<C\sqrt{t^{1/\alpha}}$. For $n-\delta_{0}>0$, by Proposition \ref{pro1}, we obtain
\begin{eqnarray*}
  III &\leq&C\Big(\frac{\sqrt{t^{1/\alpha}}}{\rho(x)}\Big)^{\delta_{0}}\Bigg(\int_{|w-y|<4|y-z|}\frac{tdw}{(\sqrt{t^{1/\alpha}}+|w-y|)^{n+2\alpha}}+
  \int_{|w-z|\leq 5|y-z|}\frac{tdw}{(\sqrt{t^{1/\alpha}}+|w-z|)^{n+2\alpha}}\Bigg)  \\
   &\leq& C\Big(\frac{\sqrt{t^{1/\alpha}}}{\rho(x)}\Big)^{\delta_{0}}\int^{5|y-z|/\sqrt{t^{1/\alpha}}}_{0}u^{n-1}du
   \leq C\Big(\frac{\sqrt{t^{1/\alpha}}}{\rho(x)}\Big)^{\delta_{0}}\Big(\frac{|y-z|}{\sqrt{t^{1/\alpha}}}\Big)^{n}  \\
   &\leq&\frac{Cr^{n}}{\rho(x)^{\delta_{0}}(\sqrt{t^{1/\alpha}})^{n-\delta_{0}}}\leq \frac{Cr^{n}}{\rho(x)^{\delta_{0}}r^{n-\delta_{0}}}=C\Big(\frac{r}{\rho(x)}\Big)^{\delta_{0}}.
\end{eqnarray*}
 Thus, when $2r<\sqrt{t^{1/\alpha}}<\rho(x)$,
$$\Big|e^{-t\mathcal{L}^{\alpha}}1(y)-e^{-t\mathcal{L}^{\alpha}}1(z)\Big|\leq C\Big(\frac{r}{\rho(x)}\Big)^{\delta}.$$
Combining the above estimates, we can get
\begin{equation}\label{4.8}
  \Big\|e^{-t\mathcal{L}^{\alpha}}1(y)-e^{-t\mathcal{L}^{\alpha}}1(z)\Big\|_{L^{\infty}((0,\infty),dt)}\leq C\Big(\frac{r}{\rho(x)}\Big)^{\delta}.
\end{equation}
Therefore, it holds
$$\log\Big(\frac{\rho(x)}{r}\Big)\frac{1}{|B|}\int_{B}\Big\|e^{-t\mathcal{L}^{\alpha}}1(y)-(e^{-t\mathcal{L}^{\alpha}}1)_{B}\Big\|_{L^{\infty}((0,\infty),dt)}dy\leq C\Big(\frac{r}{\rho(x)}\Big)^{\delta}\log\Big(\frac{\rho(x)}{r}\Big)\leq C,$$
which is the first conclusion of (iii).

For the second estimate of (iii), take $\delta\in(\gamma,\ \min\{2\alpha,\ 2-n/q\})$. by (\ref{4.8}), we have
$$\Big(\frac{\rho(x)}{r}\Big)^{\gamma}\frac{1}{|B|}\int_{B}\Big\|e^{-t\mathcal{L}^{\alpha}}1(y)-(e^{-t\mathcal{L}^{\alpha}}1)_{B}\Big\|_{L^{\infty}((0,\infty),dt)}dy\leq C\Big(\frac{r}{\rho(x)}\Big)^{\delta-\gamma}\leq C.$$
\end{proof}
\subsection{Boundedness of the Littlewood-Paley $g$-function $g^{\mathcal{L}}_{\alpha}$}\label{sec-4-2}
Similar to Section \ref{sec-4-1}, we introduce the following function space:
\begin{definition}\label{defc-2}
Let $0<\gamma\leq 1$. The Campanato type space $BMO^{\gamma}_{ \mathcal{L},L^{2}((0,\infty),dt/t)}(\mathbb{R}^{n})$ is defined as the set of all  locally integrable functions $f$ satisfying
$$\|f\|_{BMO^{\gamma}_{\mathcal{L},L^{2}((0,\infty),dt)}}:=\sup_{B\subset\mathbb{R}^{n}}
\Bigg\{\frac{1}{|B|^{1+\gamma/n}}\int_{B}\|f(x,t)-f(B,V)\|_{L^{2}((0,\infty),dt/t)}dx\Bigg\}<\infty.$$
\end{definition}
The functional calculus and the spectral theorem imply that $g^{\mathcal{L}}_{\alpha}$ is an isometry on $L^{2}(\mathbb{R}^{n})$. As before, to get the boundedness of $g^{\mathcal{L}}_{\alpha}$ on $BMO^{\gamma}_{\mathcal{L}}(\mathbb R^{n})$, it is sufficient to prove the following result.
\begin{theorem}\label{the3.1}
Assume that the potential $V\in B_{q}$ with $q>n/2$. Let $x,y,z\in\mathbb{R}^{n}$ and $N>0$.
\item{\rm (i) } For any $N>0$, there exists a constant $C_{N}$ such that
$$\Big\|D^{\mathcal{L},m}_{\alpha,t}(x,y)\Big\|_{L^{2}((0,\infty),\frac{dt}{t})}\leq \frac{C_{N}}{|x-y|^{n}}\Big(1+\frac{|x-y|}{\rho(x)}+\frac{|x-y|}{\rho(y)}\Big)^{-N}.$$
\item{\rm (ii) } If $|x-y|>2|y-z|$ and $0<\delta<\min\{2\alpha,\delta_{0},1\}$, for any constant $N>0$, there exists a constant $C_{N}$ such that
    \begin{eqnarray}\label{t-2}
      &&\Big\|D^{\mathcal{L},m}_{\alpha,t}(x,y)-D^{\mathcal{L},m}_{\alpha,t}(x,z)\Big\|_{L^{2}((0,\infty),\frac{dt}{t})}+
    \Big\|D^{\mathcal{L},m}_{\alpha,t}(y,x)-D^{\mathcal{L},m}_{\alpha,t}(z,x)\Big\|_{L^{2}((0,\infty),\frac{dt}{t})}\\
       &&\quad \leq  \frac{C_{N}|y-z|^{\delta}}{|x-y|^{n+\delta}}.\nonumber
    \end{eqnarray}
\item{\rm (iii)} There exists a constant $C$ such that for every ball $B=B(x_{0},r)$ with $0<r\leq \rho(x)/2$,
$$\log\Big(\frac{\rho(x)}{r}\Big)\frac{1}{|B|}\int_{B}\Big\|t^{m}\partial_{t}^{m}e^{-t\mathcal{L}^{\alpha}}1(y)-
(t^{m}\partial_{t}^{m}e^{-t\mathcal{L}^{\alpha}}1)_{B}\Big\|_{L^{2}((0,\infty),\frac{dt}{t})}dy\leq C,$$
and, if $\gamma<\min\{2\alpha,\delta_{0},1\}$, then
$$\Big(\frac{\rho(x)}{r}\Big)^{\gamma}\frac{1}{|B|}\int_{B}\Big\|t^{m}\partial_{t}^{m}e^{-t\mathcal{L}^{\alpha}}1(y)-(t^{m}\partial_{t}^{m}
e^{-t\mathcal{L}^{\alpha}}1)_{B}\Big\|_{L^{2}((0,\infty),\frac{dt}{t})}dy\leq C.$$
\end{theorem}
\begin{proof}
For (i), from Proposition \ref{pro3.12}, we have
$$\Big|D^{\mathcal{L},m}_{\alpha,t}(x,y)\Big|\leq C_{N}\min\Bigg\{\frac{t^{1+N/\alpha}}{|x-y|^{n+2\alpha+2N}},\ t^{-n/(2\alpha)} \Bigg\}\Big(1+\frac{\sqrt{t^{1/\alpha}}}{\rho(x)}+\frac{\sqrt{t^{1/\alpha}}}{\rho(y)}\Big)^{-N}.$$
If $\sqrt{t^{1/\alpha}}\leq |x-y|$, we obtain
\begin{eqnarray*}
   &&\Big\|D^{\mathcal{L},m}_{\alpha,t}(x,y)\Big\|_{L^{2}((0,\infty),\frac{dt}{t})}^{2}\\
   &&\leq C_{N}\int^{|x-y|^{2\alpha}}_{0}\frac{t^{2+2N/\alpha}}{|x-y|^{2n+4\alpha+4N}}
   \Big(1+\frac{\sqrt{t^{1/\alpha}}}{\rho(x)}+\frac{\sqrt{t^{1/\alpha}}}{\rho(y)}\Big)^{-2N}\frac{dt}{t} \\
   &&\leq C_{N}\Big(1+\frac{|x-y|}{\rho(x)}+\frac{|x-y|}{\rho(y)}\Big)^{-2N}\int^{|x-y|^{2\alpha}}_{0}
   \frac{(\sqrt{t^{1/\alpha}})^{4\alpha+4N}}{|x-y|^{2n+4\alpha+4N}}\Big(\frac{\sqrt{t^{1/\alpha}}}{|x-y|}\Big)^{-2N}\frac{dt}{t}  \\
   &&\leq \Big(1+\frac{|x-y|}{\rho(x)}+\frac{|x-y|}{\rho(y)}\Big)^{-2N}\frac{C_{N}}{|x-y|^{2n}}\int^{|x-y|^{2\alpha}}_{0}
   \Big(\frac{\sqrt{t^{1/\alpha}}}{|x-y|}\Big)^{4\alpha+2N}\frac{dt}{t}.
\end{eqnarray*}
Let ${\sqrt{t^{1/\alpha}}}/{|x-y|}=u$. We can see that
\begin{eqnarray*}
  \Big\|D^{\mathcal{L},m}_{\alpha,t}(x,y)\Big\|_{L^{2}((0,\infty),\frac{dt}{t})}^{2} &\leq&\Big(1+\frac{|x-y|}{\rho(x)}+\frac{|x-y|}{\rho(y)}\Big)^{-2N}\frac{C_{N}}{|x-y|^{2n}}\int^{1}_{0}u^{4\alpha+2N-1}du  \\
   &\leq&\frac{C_{N}}{|x-y|^{2n}}\Big(1+\frac{|x-y|}{\rho(x)}+\frac{|x-y|}{\rho(y)}\Big)^{-2N}.
\end{eqnarray*}
If $\sqrt{t^{1/\alpha}}\geq |x-y|$, we can get
\begin{eqnarray*}
 \Big\|D^{\mathcal{L},m}_{\alpha,t}(x,y)\Big\|^{2}_{L^{2}((0,\infty),\frac{dt}{t})}  &\leq&C_{N}\Big(1+\frac{|x-y|}{\rho(x)}+\frac{|x-y|}{\rho(y)}\Big)^{-2N}\int^{\infty}_{|x-y|^{2\alpha}}t^{-n/\alpha-1}dt  \\
   &\leq& \frac{C_{N}}{|x-y|^{2n}}\Big(1+\frac{|x-y|}{\rho(x)}+\frac{|x-y|}{\rho(y)}\Big)^{-2N}.
\end{eqnarray*}

For (ii), By Proposition \ref{pro3.12}, we have
\begin{eqnarray*}
 \Big\|D^{\mathcal{L},m}_{\alpha,t}(x,y)-D^{\mathcal{L},m}_{\alpha,t}(x,z)\Big\|^{2}_{L^{2}((0,\infty),\frac{dt}{t})}  &\leq& \int^{\infty}_{0}\frac{C_{N}t}{(\sqrt{t^{1/\alpha}}+|x-y|)^{2n+4\alpha}}\Big(
 \frac{|y-z|}{\sqrt{t^{1/\alpha}}}\Big)^{2\delta}dt \\
   &\leq& C_{N}|y-z|^{2\delta}\int^{\infty}_{0}\frac{(\sqrt{t^{1/\alpha}})^{2\alpha-2\delta}}{(\sqrt{t^{1/\alpha}}+|x-y|)^{2n+4\alpha}}dt.
\end{eqnarray*}
Let ${\sqrt{t^{1/\alpha}}}/{|x-y|}=u$. We obtain
\begin{eqnarray*}
  \Big\|D^{\mathcal{L},m}_{\alpha,t}(x,y)-D^{\mathcal{L},m}_{\alpha,t}(x,z)\Big\|^{2}_{L^{2}((0,\infty),\frac{dt}{t})} &\leq&C_{N}|y-z|^{2\delta}|x-y|^{-2n-2\delta}\int^{\infty}_{0}\frac{u^{4\alpha-2\delta-1}}{(1+u)^{2n+4\alpha}}du  \\
   &\leq& \frac{C_{N}|y-z|^{2\delta}}{|x-y|^{2n+2\delta}}.
\end{eqnarray*}
The symmetry of the kernel $D^{L,m}_{\alpha,t}(\cdot,\cdot)$ gives the conclusion of (ii).

For (iii), let us fix $y,z\in B=B(x_{0},r)$, $0<r\leq \rho(x_{0})/2$. Similar to Theorem \ref{the3}, we must handle $$\Big\|t^{m}\partial_{t}^{m}e^{-t\mathcal{L}^{\alpha}}1(y)-t^{m}\partial_{t}^{m}e^{-t\mathcal{L}^{\alpha}}1(z)\Big\|_{L^{2}((0,\infty),\frac{dt}{t})}.$$ We can write
\begin{eqnarray*}
  \Big\|t^{m}\partial_{t}^{m}e^{-t\mathcal{L}^{\alpha}}1(y)-t^{m}\partial_{t}^{m}e^{-t\mathcal{L}^{\alpha}}1(z)\Big\|^{2}_{L^{2}((0,\infty),\frac{dt}{t})}
   &=& M_{1}+M_{2}+M_{3},
\end{eqnarray*}
where
\begin{equation*}
\left\{  \begin{aligned}
 M_{1}&:=\int^{(2r)^{2\alpha}}_{0}\Big|
   \int_{\mathbb{R}^{n}}\Big(D^{\mathcal{L},m}_{\alpha,t}(x,y)-D^{\mathcal{L},m}_{\alpha,t}(x,z)\Big)dx\Big|^{2}\frac{dt}{t};\\
 M_{2}&:=\int^{\rho(x_{0})^{2\alpha}}_{(2r)^{2\alpha}}\Big|
   \int_{\mathbb{R}^{n}}\Big(D^{\mathcal{L},m}_{\alpha,t}(x,y)-D^{\mathcal{L},m}_{\alpha,t}(x,z)\Big)dx\Big|^{2}\frac{dt}{t};\\
 M_{3}&:=\int^{\infty}_{\rho(x_{0})^{2\alpha}}\Big|
   \int_{\mathbb{R}^{n}}\Big(D^{\mathcal{L},m}_{\alpha,t}(x,y)-D^{\mathcal{L},m}_{\alpha,t}(x,z)\Big)dx\Big|^{2}\frac{dt}{t}.
  \end{aligned}\right.
\end{equation*}

Since $y,z\in B\subset B(x_{0},\rho(x_{0}))$, it follows that $\rho(y)\sim\rho(x_{0})\sim\rho(z)$. By Proposition \ref{pro3.12} (iii),
$$M_{1}\leq C\int^{(2r)^{2\alpha}}_{0}\frac{(\sqrt{t^{1/\alpha}}/\rho(x_{0}))^{2\delta}}{(1+\sqrt{t^{1/\alpha}}/\rho(x_{0}))^{2N}}\frac{dt}{t}
\leq C\int^{(2r)^{2\alpha}}_{0}\Big(\frac{\sqrt{t^{1/\alpha}}}{\rho(x_{0})}\Big)^{2\delta}\frac{dt}{t}=C\Big(\frac{r}{\rho(x_{0})}\Big)^{2\delta}.$$
Also, by Proposition \ref{pro3.12}(ii),
\begin{eqnarray*}
  M_{3} &\leq& C\int^{\infty}_{\rho(x_{0})^{2\alpha}}\Big(\frac{|y-z|}{\sqrt{t^{1/\alpha}}}\Big)^{2\delta}\Big|\int_{\mathbb{R}^{n}}
  \frac{t}{(\sqrt{t^{1/\alpha}}+|x-y|)^{n+2\alpha}}dx\Big|^{2}\frac{dt}{t}  \\
   &\leq& C\int^{\infty}_{\rho(x_{0})^{2\alpha}}\Big(\frac{|y-z|}{\sqrt{t^{1/\alpha}}}\Big)^{2\delta}\frac{dt}{t}\leq C\Big(\frac{r}{\rho(x_{0})}\Big)^{2\delta}.
\end{eqnarray*}
It remains to estimate the term $M_{2}$. In this case, $|y-z|\leq 2r\leq \sqrt{t^{1/\alpha}}\leq \rho(x_{0})$. Then we can use the methods in Theorem \ref{the3},
\begin{eqnarray*}
  M_{2}
   =\int^{\rho(x_{0})^{2\alpha}}_{(2r)^{2\alpha}}\Big|M_{2,1}+M_{2,2}+M_{2,3}\Big|^{2}\frac{dt}{t},
\end{eqnarray*}
where
\begin{equation*}
 \left\{ \begin{aligned}
  M_{2,1}&:=\int_{|x-y|>c\rho(y)>4|y-z|}\Big(D^{\mathcal{L},m}_{\alpha,t}(x,y)-D^{m}_{\alpha,t}(x-y)\Big)-\Big(D^{\mathcal{L},m}_{\alpha,t}(x,z)-D^{m}_{\alpha,t}(x-z)\Big)dx;\\
  M_{2,2}&:=\int_{4|y-z|<|x-y|<c\rho(y)}\Big(D^{\mathcal{L},m}_{\alpha,t}(x,y)-D^{m}_{\alpha,t}(x-y)\Big)-\Big(D^{\mathcal{L},m}_{\alpha,t}(x,z)-D^{m}_{\alpha,t}(x-z)\Big)dx;\\
  M_{2,3}&:=\int_{|x-y|<4|y-z|}\Big(D^{\mathcal{L},m}_{\alpha,t}(x,y)-D^{m}_{\alpha,t}(x-y)\Big)-\Big(D^{\mathcal{L},m}_{\alpha,t}(x,z)-D^{m}_{\alpha,t}(x-z)\Big)dx.
\end{aligned}\right.
\end{equation*}

For $M_{2,1}$, similar to prove (\ref{t-2}), we also can get
$$\Big|D^{\mathcal{L},m}_{\alpha,t}(x,y)-D^{\mathcal{L},m}_{\alpha,t}(x,z)\Big|\leq \frac{C|y-z|^{\delta}}{|x-y|^{n+\delta}},$$
which is valid to $D^{m}_{\alpha,t}(\cdot)$. So we obtain
$$|M_{2,1}|\leq C\int_{|x-y|>c\rho(y)>4|y-z|}\frac{|y-z|^{\delta}}{|x-y|^{n+\delta}}dx\leq C\Big(\frac{r}{\rho(x_{0})}\Big)^{\delta}.$$
For $M_{2,2}$, by Proposition \ref{pro2-2-2}
and the fact that $\rho(x)\sim \rho(y)$ in the region of integration.
\begin{eqnarray*}
  |M_{2,2}| &\leq&C|y-z|^{\delta}\int_{4|y-z|<|x-y|<c\rho(y)}\frac{t}{\rho(x)^{\delta}(\sqrt{t^{1/\alpha}}+|x-y|)^{n+2\alpha}}dx \leq C\Big(\frac{r}{\rho(x_{0})}\Big)^{\delta}.
\end{eqnarray*}

For $M_{2,3}$, since $|y-z|\leq 2r<\sqrt{t^{1/\alpha}},$ we have $|x-y|<C\sqrt{t^{1/\alpha}}$. For $n-\delta_{0}>0$, by Proposition \ref{pro1-1-1},
we obtain
\begin{eqnarray*}
 | M_{2,3}| &\leq&C\Big(\frac{\sqrt{t^{1/\alpha}}}{\rho(x_{0})}\Big)^{\delta_{0}}\Bigg(\int_{|x-y|<4|y-z|}\frac{tdx}{(\sqrt{t^{1/\alpha}}+|x-y|)^{n+2\alpha}}+
  \int_{|x-z|\leq 5|y-z|}\frac{tdx}{(\sqrt{t^{1/\alpha}}+|x-z|)^{n+2\alpha}}\Bigg)  \\
   &\leq& C\Big(\frac{\sqrt{t^{1/\alpha}}}{\rho(x_{0})}\Big)^{\delta_{0}}\int^{5|y-z|/\sqrt{t^{1/\alpha}}}_{0}u^{n-1}du  \\
   &\leq&C\Big(\frac{\sqrt{t^{1/\alpha}}}{\rho(x_{0})}\Big)^{\delta_{0}}\Big(\frac{|y-z|}{\sqrt{t^{1/\alpha}}}\Big)^{n}  \\
   &\leq&\frac{Cr^{n}}{\rho(x_{0})^{\delta_{0}}(\sqrt{t^{1/\alpha}})^{n-\delta_{0}}}\leq C\Big(\frac{r}{\rho(x_{0})}\Big)^{\delta_{0}}.
\end{eqnarray*}
The estimates for $M_{2,i}, i=1,2,3$, imply that
$$M_{2}\leq \int^{\rho(x_{0})^{2\alpha}}_{(2r)^{2\alpha}}\Big(\frac{r}{\rho(x_{0})}\Big)^{2\delta}\frac{dt}{t}=C\Big(\frac{r}{\rho(x_{0})}\Big)^{2\delta}
\log\Big(\frac{\rho(x_{0})}{r}\Big).$$
Finally, we can get
$$\Big\|t^{m}\partial_{t}^{m}e^{-t\mathcal{L}^{\alpha}}1(y)-t^{m}\partial_{t}^{m}e^{-t\mathcal{L}^{\alpha}}1(z)\Big\|_{L^{2}((0,\infty),\frac{dt}{t})}\leq C\Big(\frac{r}{\rho(x_{0})}\Big)^{\delta}\Big(\log(\frac{\rho(x_{0})}{r})\Big)^{1/2}.$$
Thus (iii) readily follows.

\end{proof}

\subsection{Boundedness of Littlewood-Paley $g$-function  $\widetilde{g}^{\mathcal{L}}_{\alpha}$}\label{sec-4-2}
 By the $L^{2}$-boundedness of Riesz transforms $\nabla_{x}\mathcal{L}^{-1/2}$, we can see that
 $$\|\widetilde{g}^{\mathcal{L}}_{\alpha}f\|^{2}_{L^{2}}\leq C\int^{\infty}_{0}\Big(\int_{\mathbb{R}^{n}}|t^{1/2\alpha}\mathcal{L}^{1/2}e^{-t\mathcal{L}^{\alpha}}f(x)|^{2}dx\Big)\frac{dt}{t}.$$
 Then by the spectral theorem, we know that $\widetilde{g}^{\mathcal{L}}_{\alpha}$ is bounded from $L^{2}(\mathbb{R}^{n})$ into $L^{2}(\mathbb{R}^{n})$.

\begin{theorem}\label{the3.2}
Assume that the potential $V\in B_{q}$ with $q>n$. Let $x,y,z\in\mathbb{R}^{n}$. Then
\item{\rm (i) } For any $N>0$, there exists a constant $C_{N}$ such that
$$\Big\|\widetilde{D}^{\mathcal{L}}_{\alpha,t}(x,y)\Big\|_{L^{2}((0,\infty),\frac{dt}{t})}\leq \frac{C_{N}}{|x-y|^{n}}\Big(1+\frac{|x-y|}{\rho(x)}+\frac{|x-y|}{\rho(y)}\Big)^{-N};$$
\item{\rm (ii) } Let  $|x-y|>2|y-z|$ and $0<\delta<\min\{2\alpha,\delta_{1},1\}$. For any $N>0$, there exists a constant $C_{N}$ such that $$\Big\|\widetilde{D}^{\mathcal{L}}_{\alpha,t}(x,y)-\widetilde{D}^{\mathcal{L}}_{\alpha,t}(x,z)\Big\|_{L^{2}((0,\infty),\frac{dt}{t})}+
    \Big\|\widetilde{D}^{\mathcal{L}}_{\alpha,t}(y,x)-\widetilde{D}^{\mathcal{L}}_{\alpha,t}(z,x)\Big\|_{L^{2}((0,\infty),\frac{dt}{t})}\leq
\frac{C_{N}|y-z|^{\delta}}{|x-y|^{n+\delta}};$$
\item{\rm (iii)} There exists a constant $C$ such that for every ball $B=B(x_{0},r)$ with $0<r\leq \rho(x)/2$,
$$\log\Big(\frac{\rho(x)}{r}\Big)\frac{1}{|B|}\int_{B}\Big\|t^{1/(2\alpha)}\nabla_{x}
e^{-t\mathcal{L}^{\alpha}}1(y)-(t^{1/(2\alpha)}\nabla_{x}e^{-t\mathcal{L}^{\alpha}}1)_{B}\Big\|_{L^{2}((0,\infty),\frac{dt}{t})}dy\leq C,$$
and, if $\gamma<\min\{2\alpha,\delta_{1},1\}$ then
$$\Big(\frac{\rho(x)}{r}\Big)^{\gamma}\frac{1}{|B|}\int_{B}\Big\|t^{1/(2\alpha)}\nabla_{x}
e^{-t\mathcal{L}^{\alpha}}1(y)-(t^{1/(2\alpha)}\nabla_{x}e^{-t\mathcal{L}^{\alpha}}1)_{B}\Big\|_{L^{2}((0,\infty),\frac{dt}{t})}dy\leq C.$$
\end{theorem}
\begin{proof}
For (i), from Proposition \ref{pro3.13}, we have
$$\Big|\widetilde{D}^{\mathcal{L}}_{\alpha,t}(x,y)\Big|\leq C_{N}\min\Bigg\{\frac{t^{1+N/\alpha+1/(2\alpha)}}{|x-y|^{n+2\alpha+2N+1}},\ t^{-n/(2\alpha)} \Bigg\}\Big(1+\frac{\sqrt{t^{1/\alpha}}}{\rho(x)}+\frac{\sqrt{t^{1/\alpha}}}{\rho(y)}\Big)^{-N}.$$
If $\sqrt{t^{1/\alpha}}\leq |x-y|$, we obtain
\begin{eqnarray*}
   \Big\|\widetilde{D}^{\mathcal{L}}_{\alpha,t}(x,y)\Big\|_{L^{2}((0,\infty),\frac{dt}{t})}^{2}&\leq& C_{N}\int^{|x-y|^{2\alpha}}_{0}\frac{t^{2+2N/\alpha+1/\alpha}}{|x-y|^{2n+4\alpha+4N+2}}
   \Big(1+\frac{\sqrt{t^{1/\alpha}}}{\rho(x)}+\frac{\sqrt{t^{1/\alpha}}}{\rho(y)}\Big)^{-2N}\frac{dt}{t} \\
   &\leq&C_{N}\Big(1+\frac{|x-y|}{\rho(x)}+\frac{|x-y|}{\rho(y)}\Big)^{-2N}\int^{|x-y|^{2\alpha}}_{0}
   \frac{(\sqrt{t^{1/\alpha}})^{4\alpha+4N+2}}{|x-y|^{2n+4\alpha+4N+2}}\Big(\frac{\sqrt{t^{1/\alpha}}}{|x-y|}\Big)^{-2N}\frac{dt}{t}  \\
   &\leq& \Big(1+\frac{|x-y|}{\rho(x)}+\frac{|x-y|}{\rho(y)}\Big)^{-2N}\frac{C_{N}}{|x-y|^{2n}}\int^{|x-y|^{2\alpha}}_{0}
   \Big(\frac{\sqrt{t^{1/\alpha}}}{|x-y|}\Big)^{4\alpha+2N+2}\frac{dt}{t}.
\end{eqnarray*}
Let $\sqrt{t^{1/\alpha}}/|x-y|=u$. We can see that
\begin{eqnarray*}
  \Big\|\widetilde{D}^{\mathcal{L}}_{\alpha,t}(x,y)\Big\|_{L^{2}((0,\infty),\frac{dt}{t})}^{2} &\leq&\Big(1+\frac{|x-y|}{\rho(x)}+\frac{|x-y|}{\rho(y)}\Big)^{-2N}\frac{C_{N}}{|x-y|^{2n}}\int^{1}_{0}u^{4\alpha+2N+1}du  \\
   &\leq&\frac{C_{N}}{|x-y|^{2n}}\Big(1+\frac{|x-y|}{\rho(x)}+\frac{|x-y|}{\rho(y)}\Big)^{-2N}.
\end{eqnarray*}
If $\sqrt{t^{1/\alpha}}\geq |x-y|$, we can get
\begin{eqnarray*}
 \Big\|\widetilde{D}^{\mathcal{L}}_{\alpha,t}(x,y)\Big\|^{2}_{L^{2}((0,\infty),\frac{dt}{t})}  &\leq&C_{N}\Big(1+\frac{|x-y|}{\rho(x)}+\frac{|x-y|}{\rho(y)}\Big)^{-2N}\int^{\infty}_{|x-y|^{2\alpha}}t^{-n/\alpha-1}dt  \\
   &\leq& \frac{C_{N}}{|x-y|^{2n}}\Big(1+\frac{|x-y|}{\rho(x)}+\frac{|x-y|}{\rho(y)}\Big)^{-2N},
\end{eqnarray*}
which proves (i).

For (ii), By Proposition \ref{pro3.13}, we have
\begin{eqnarray*}
 \Big\|\widetilde{D}^{\mathcal{L}}_{\alpha,t}(x,y)-\widetilde{D}^{\mathcal{L}}_{\alpha,t}(x,z)\Big\|^{2}_{L^{2}((0,\infty),\frac{dt}{t})}  &\leq& \int^{\infty}_{0}\frac{C_{N}t}{(\sqrt{t^{1/\alpha}}+|x-y|)^{2n+4\alpha}}\Big(
 \frac{|y-z|}{\sqrt{t^{1/\alpha}}}\Big)^{2\delta}dt \\
   &\leq& C_{N}|y-z|^{2\delta}\int^{\infty}_{0}\frac{(\sqrt{t^{1/\alpha}})^{2\alpha-2\delta}}{(\sqrt{t^{1/\alpha}}+|x-y|)^{2n+4\alpha}}dt.
\end{eqnarray*}
Let ${\sqrt{t^{1/\alpha}}}/{|x-y|}=u$. We obtain
\begin{eqnarray*}
  \Big\|\widetilde{D}^{\mathcal{L}}_{\alpha,t}(x,y)-\widetilde{D}^{\mathcal{L}}_{\alpha,t}(x,z)\Big\|^{2}_{L^{2}((0,\infty),\frac{dt}{t})} &\leq&C_{N}|y-z|^{2\delta}|x-y|^{-2n-2\delta}\int^{\infty}_{0}\frac{u^{4\alpha-2\delta-1}}{(1+u)^{2n+4\alpha}}du  \\
   &\leq& \frac{C_{N}|y-z|^{2\delta}}{|x-y|^{2n+2\delta}}.
\end{eqnarray*}
The symmetry of the kernel $D^{\mathcal{L}}_{\alpha,t}(\cdot,\cdot)$ gives the desired conclusion of (ii).

For (iii), fix $y,z\in B=B(x_{0},r)$ with $0<r\leq \rho(x_{0})/2$. Similar to Theorem \ref{the3.1}, we need to deal with the term $$\Big\|t^{1/(2\alpha)}\nabla_{x}e^{-t\mathcal{L}^{\alpha}}1(y)-t^{1/(2\alpha)}\nabla_{x}e^{-t\mathcal{L}^{\alpha}}1(z)\Big\|_{L^{2}((0,\infty),\frac{dt}{t})}$$ first. Write
\begin{eqnarray*}
  \Big\|t^{1/(2\alpha)}\nabla_{x}e^{-t\mathcal{L}^{\alpha}}1(y)-t^{1/(2\alpha)}\nabla_{x}e^{-t\mathcal{L}^{\alpha}}1(z)\Big\|^{2}_{L^{2}((0,\infty),\frac{dt}{t})}
  &=& G_{1}+G_{2}+G_{3},
\end{eqnarray*}
where
\begin{equation*}
 \left\{ \begin{aligned}
 G_{1}&:=\int^{(2r)^{2\alpha}}_{0}\Big|\int_{\mathbb{R}^{n}}\Big(\widetilde{D}^{\mathcal{L}}_{\alpha,t}(x,y)-\widetilde{D}^{\mathcal{L}}_{\alpha,t}(x,z)\Big)dx\Big|^{2}\frac{dt}{t};\\
 G_{2}&:=\int^{\rho(x_{0})^{2\alpha}}_{(2r)^{2\alpha}}\Big|\int_{\mathbb{R}^{n}}\Big(\widetilde{D}^{\mathcal{L}}_{\alpha,t}(x,y)-\widetilde{D}^{\mathcal{L}}_{\alpha,t}(x,z)\Big)dx\Big|^{2}\frac{dt}{t};\\
 G_{3}&:=\int^{\infty}_{\rho(x_{0})^{2\alpha}}\Big|\int_{\mathbb{R}^{n}}\Big(\widetilde{D}^{\mathcal{L}}_{\alpha,t}(x,y)-\widetilde{D}^{\mathcal{L}}_{\alpha,t}(x,z)\Big)dx\Big|^{2}\frac{dt}{t}. \end{aligned}\right.
\end{equation*}

Since $y,z\in B\subset B(x_{0},\rho(x_{0}))$, then $\rho(y)\sim\rho(x_{0})\sim\rho(z)$. It follows from Proposition \ref{pro3.13} (iii) that
$$G_{1}\leq C\int^{(2r)^{2\alpha}}_{0}(\sqrt{t^{1/\alpha}}/\rho(x_{0}))^{1+2\alpha}\frac{dt}{t}
\leq C\int^{(2s)^{2\alpha}}_{0}\Big(\frac{\sqrt{t^{1/\alpha}}}{\rho(x_{0})}\Big)^{2\delta}\frac{dt}{t}=C\Big(\frac{r}{\rho(x_{0})}\Big)^{2\delta}.$$
Also, we apply Proposition \ref{pro3.13} (ii) to deduce that
\begin{eqnarray*}
  G_{3} &\leq& C\int^{\infty}_{\rho(x_{0})^{2\alpha}}\Big(\frac{|y-z|}{\sqrt{t^{1/\alpha}}}\Big)^{2\delta}\Big|\int_{\mathbb{R}^{n}}
  \frac{t}{(\sqrt{t^{1/\alpha}}+|x-y|)^{n+2\alpha}}dx\Big|^{2}\frac{dt}{t}  \\
   &\leq& C\int^{\infty}_{\rho(x_{0})^{2\alpha}}\Big(\frac{|y-z|}{\sqrt{t^{1/\alpha}}}\Big)^{2\delta}\frac{dt}{t}\leq C\Big(\frac{r}{\rho(x_{0})}\Big)^{2\delta}.
\end{eqnarray*}
Then for $G_{2}$, following the procedure of the treatment for $M_{2}$ in Theorem \ref{the3.1}, we obtain
$$G_{2}\leq \int^{\rho(x_{0})^{2\alpha}}_{(2r)^{2\alpha}}\Big(\frac{r}{\rho(x_{0})}\Big)^{2\delta}\frac{dt}{t}=C\Big(\frac{r}{\rho(x_{0})}\Big)^{2\delta}
\log\Big(\frac{\rho(x_{0})}{r}\Big).$$

From the above estimates, we can get
$$\Big\|t^{1/(2\alpha)}\nabla_{x}e^{-t\mathcal{L}^{\alpha}}1(y)-t^{1/(2\alpha)}\nabla_{x}e^{-t\mathcal{L}^{\alpha}}1(z)\Big\|^{2}_{L^{2}((0,\infty),\frac{dt}{t})}
\leq C\Big(\frac{r}{\rho(x_{0})}\Big)^{\delta}\Big(\log(\frac{\rho(x_{0})}{r})\Big)^{1/2}.$$
Thus (iii) readily follows.

\end{proof}


%
%



\end{document}